\documentclass[11pt]{article}

\usepackage{amssymb}

\newcommand{\qed}{$\;\;\;\Box$}
\newenvironment{proof}{\par\smallbreak{\sl Proof.~}}
{\unskip\nobreak\hfill \qed \par\medbreak}

\newcounter{claim}
\renewcommand{\theclaim}{\arabic{claim}}
{\par\medskip\par}

{\qed\par\smallbreak}
\newcommand{\hide}[1]{}

\newcommand{\Con}{{\mbox{C}}}

\newcommand{\N}{{\mathbb N}}

\newcommand{\R}{{\mathbb R}}
\newcommand{\C}{{\mathbb C}}
\newcommand{\Z}{{\mathbb Z}}
\newcommand{\M}{{\mathbb M}}

\newcommand{\Aa}{{\mathbb A}}
\newcommand{\Bb}{{\mathbb B}}

\newcommand{\A}{{\cal A}}
\newcommand{\B}{{\cal B}}
\newcommand{\CC}{{\cal C}}

\newcommand{\LL}{{\cal L}}
\newcommand{\beq}{\begin{equation}}
\newcommand{\ee}{\end{equation}}

\renewcommand{\d}{\partial}

\newtheorem{thm}{Theorem}[section]
\newtheorem{lemma}[thm]{Lemma}

\newtheorem{cor}[thm]{Corollary}
\newtheorem{rem}[thm]{Remark}

\newcommand{\al}{\alpha}
\newcommand{\be}{\beta}
\newcommand{\ga}{\gamma}
\newcommand{\de}{\delta}
\newcommand{\eps}{\varepsilon}
\newcommand{\vphi}{\varphi}

\newcommand{\om}{\omega}

\newcommand{\reff}[1]{(\ref{#1})}

\newcommand{\im}{\mathop{\rm im}}

\newcommand{\diag}{\mathop{\rm diag}\nolimits}

\newcommand{\ess}{\mathop{\rm ess}}

\newtheorem{ex}[thm]{Example}

\title{
Fredholmness and Smooth Dependence for Linear Time-Periodic Hyperbolic Systems
} 

\newcounter{thesame}
\author{
I.~Kmit
 \ \ \ L.~Recke\\
{\small
Institute of Mathematics, Humboldt University of Berlin,}
\\
{\small Rudower Chaussee 25, D-12489 Berlin, Germany }
\\
{\small
and Institute for Applied Problems of Mechanics and Mathematics, }
\\
{\small
Ukrainian Academy of Sciences,  Naukova St.\ 3b, 79060 Lviv,
Ukraine 
}
\\
{\small   E-mail:
{\tt kmit@informatik.hu-berlin.de}}\\[5mm]
{\small
Institute of Mathematics, Humboldt University of Berlin,}\\
{\small 
Rudower Chaussee 25, D-12489 Berlin, Germany}\\
{\small   E-mail:
{\tt recke@mathematik.hu-berlin.de}}
}
\date{}

\begin{document}

\maketitle

\begin{abstract}
This paper concerns $n\times n$ linear one-dimensional hyperbolic systems  
of the type 
$$
\partial_tu_j  + a_j(x)\partial_xu_j + \sum\limits_{k=1}^nb_{jk}(x)u_k = f_j(x,t),\; j=1,\ldots,n,
$$
with periodicity conditions in time and reflection boundary conditions in space. 
We state conditions on the data $a_j$ and  $b_{jk}$ and the reflection coefficients
such that the system is Fredholm solvable.
Moreover, we state  conditions on the data such that for any right hand side there exists exactly
one solution, that the solution survives under small perturbations of the data, and that the corresponding 
data-to-solution-map
is smooth with respect to appropriate function space norms. In particular, those  conditions
imply that no small denominator effects occur.

We  show that perturbations of the coefficients $a_j$ lead to 
essentially different results than  perturbations of the coefficients $b_{jk}$, in general.
Our results
cover cases of non-strictly hyperbolic systems as well as systems with discontinuous coefficients $a_j$
and $b_{jk}$, but they are new even  in the case of strict hyperbolicity and of smooth coefficients.
\end{abstract}


\section{Introduction}\label{sec:intr}
\renewcommand{\theequation}{{\thesection}.\arabic{equation}}
\setcounter{equation}{0}

\subsection{Problem and main results}\label{sec:results}

This paper concerns linear inhomogeneous hyperbolic systems of first order PDEs
in one space dimension of the type
\beq\label{eq:1.1}
\partial_tu_j  + a_j(x)\partial_xu_j + \sum\limits_{k=1}^nb_{jk}(x)u_k = f_j(x,t), \;
j=1,\ldots,n,\;  x\in(0,1)  
\ee
with time-periodicity conditions
\beq\label{eq:1.2}
u_j(x,t+2\pi) = u_j(x,t),\; j=1,\ldots,n,\;  x\in[0,1]  
\ee
and reflection boundary conditions
\beq\label{eq:1.3}
\begin{array}{l}
\displaystyle
u_j(0,t) = \sum\limits_{k=m+1}^nr_{jk}^0u_k(0,t),\; j=1,\ldots,m,\\
\displaystyle
u_j(1,t) = \sum\limits_{k=1}^mr_{jk}^1u_k(1,t), \; j=m+1,\ldots,n.\\
\end{array}
\ee
Here $1\le m<n$ are fixed natural numbers, 
$r_{jk}^0$ and $r_{jk}^1$ are real numbers, 
and the right-hand sides $f_j:[0,1] \times \R \to \R$
are supposed to be $2\pi$-periodic with respect to $t$.

Roughly speaking, we will prove results of the following type:

First, we will state sufficient conditions on the data $a_j, b_{jk},r_{jk}^0$, and  $r_{jk}^1$
such that the system~(\ref{eq:1.1})--(\ref{eq:1.3}) has a Fredholm type solution behavior, i.e. that
it is solvable if and only if the right hand side is orthogonal to all solutions 
to the corresponding homogeneous adjoint 
system
$$
-\partial_tu_j  - \partial_x\left(a_j(x)u_j\right) + \sum\limits_{k=1}^nb_{kj}(x)u_k = 0,
\; j=1,\ldots,n,\;  x\in(0,1),  
$$
$$
u_j(x,t+2\pi) = u_j(x,t),\; j=1,\ldots,n,\;  x\in[0,1], 
$$
\beq \label{eq:1.6}
\begin{array}{l}
\displaystyle
a_j(0)u_j(0,t) = -\sum\limits_{k=1}^mr_{kj}^0a_k(0)u_k(0,t),\; j=m+1,\ldots,n,\\
\displaystyle
a_j(1)u_j(1,t) = -\sum\limits_{k=m+1}^nr_{kj}^1a_k(1)u_k(1,t),\; j=1,\ldots,m.
\end{array}
\end{equation}
And second, we will state sufficient conditions on the data $a_j, b_{jk},r_{jk}^0$, and  $r_{jk}^1$
such that the system~(\ref{eq:1.1})--(\ref{eq:1.3}) is uniquely solvable for any  right hand side,
that this unique solvability property survives under small perturbations of the data, and that the corresponding 
data-to-solution-maps
are smooth  with respect to appropriate function space norms.
For example, under those sufficient conditions the following is true: \\

(I) If  $\partial_t^j f \in L^2\left((0,1)\times(0,2\pi);\R^n\right)$ for $j=0,1$,
then the map $b \mapsto u$ is $C^\infty$-smooth from an open set in
$L^\infty\left((0,1);\M_n\right)$ into $L^2\left((0,1)\times(0,2\pi);\R^n\right)$.\\

(II) If  $\partial_t^j f \in L^2\left((0,1)\times(0,2\pi);\R^n\right)$ for $j=0,1,2$,
then the map $b \mapsto u$ is $C^\infty$-smooth from an open set in
$L^\infty\left((0,1);\M_n\right)$ into $C\left([0,1]\times[0,2\pi];\R^n\right)$.\\

(III) If  $\partial_t^j f \in L^2\left((0,1)\times(0,2\pi);\R^n\right)$ for $j=0,1, 
\ldots,k$ with $k \ge 2$, then the map $a \mapsto u$ is $C^{k-1}$-smooth 
(rsp.  $C^{k-2}$-smooth)
from an open subset of
$BV\left((0,1);\M_n\right)$ into  $L^2\left((0,1)\times(0,2\pi);\R^n\right)$.
(rsp.  $C\left([0,1]\times[0,2\pi];\R^n\right)$).
\\

\noindent
Here  and in what follows we denote by
$$
a:=\diag(a_1,\ldots,a_n), \;  b:=[b_{jk}]_{j,k=1}^n, f:=(f_1,\ldots,f_n),  \mbox{ and } u:=(u_1,\ldots,u_n)
$$
the diagonal matrix of the  coefficient functions $a_j$, the matrix of  the  coefficient functions $b_{jk}$,
and the vectors of  the right hand sides $f_j$ and the solutions $u_j$, respectively, and  $\M_n$ 
is the space of all real $n \times n$ matrices.

In order to formulate our results more precisely, let us introduce the following
function spaces: 
For $\ga\ge 0$ we denote by $W^{\ga}$ the vector space of all locally integrable functions
$f: [0,1]\times\R\to\R^n$ such that
$f(x,t)=f\left(x,t+2\pi\right)$ for almost all $x \in (0,1)$ and $t\in\R$
and that
\begin{equation}\label{eq:1.12}
\|f\|_{W^{\ga}}^2:=\sum\limits_{s\in\Z}(1+s^2)^{\gamma}
\int\limits_0^1\left\|\int\limits_0^{2\pi}
f(x,t)e^{-ist}\,dt\right\|^2\,dx<\infty.
\end{equation}
Here and in what follows $\|\cdot\|$ is the Hermitian norm in $\C^n$. 
It is well-known
(see, e.g., \cite{herrmann}, \cite[Chapter 5.10]{robinson}, 
and \cite[Chapter 2.4]{vejvoda}) 
that $W^{\ga}$ is a Banach space 
with the norm~(\ref{eq:1.12}). 
In fact, it is the anisotropic Sobolev space of all  measurable functions
$u: [0,1]\times\R\to\R^n$ such that $u(x,t)=u\left(x,t+2\pi\right)$ for almost all  $x \in (0,1)$ and $t\in\R$
and that the distributional partial derivatives of $u$ with respect to $t$ up to the order $\ga$ are  locally 
quadratically integrable.

Further,  for $\ga\ge 1$  and $a \in L^\infty\left((0,1);\M_n\right)$ with $\mbox{ess inf } |a_j|>0$ 
for all $j=1,\ldots,n$
we will work with the
function spaces
$$
U^{\gamma}(a) := \Bigl\{u\in W^{\gamma}:\, \d_xu\in W^{0},\,
\d_tu+a\d_xu\in W^{\gamma}\Bigr\} 
$$
endowed with the norms
$$
\|u\|_{U^{\gamma}(a)}^2:=\|u\|_{W^{\gamma}}^2
+\left\|\d_tu+a\d_xu\right\|_{W^{\gamma}}^2.
$$
Remark that the space $U^{\gamma}(a)$ depends on $a$. 
In particular, it is larger than the space of  all $u \in W^\ga$ such that
$\partial_t u \in  W^\ga$ and $\partial_x u \in  W^\ga$ (which does not  depend on $a$).
For $u \in U^\ga(a)$ there exist traces $u(0,\cdot), u(1,\cdot) \in
L^2_{loc}(\R;\R^n)$ (see Section \ref{sec:spaces}), and,  hence, it makes sense
to consider the closed subspaces in $U^\ga(a)$
\begin{eqnarray*}
V^\ga(a,r)&:=&\{u \in U^\ga(a):\, (\ref{eq:1.3}) \mbox{ is fulfilled}\},\\
\tilde{V}^\ga(a,r)&:=&\{u \in U^\ga(a):\, (\ref{eq:1.6}) \mbox{ is fulfilled}\}.
\end{eqnarray*}
Here we use the notation
$$
r:=(r^0,r^1) \mbox{ with } r^0:=[r^0_{jk}]_{j=1,k=m+1}^{m\;\;\;\;\;n}\;, 
r^1:=[r^1_{jk}]_{j=m+1,k=1}^{n\;\;\;\;\;\;\;\;\;\;\;m}
$$
for the matrices of  the  reflection coefficients $r^0_{jk}$ and  $r^1_{jk}$.
Further,  we denote by
$$
b^0:=\mbox{diag}(b_{11},b_{22},\ldots,b_{nn}) \; \mbox{ and } \; b^1:=b-b^0
$$
the diagonal and the off-diagonal parts of the coefficient matrix $b$, respectively.

Further, we introduce 
operators $\A(a,b^0)\in\LL(V^{\gamma}(a,r);W^{\gamma})$,
$\tilde{\A}(a,b^0)\in\LL(\tilde{V}^{\gamma}(a,r);W^{\gamma})$
and $\B(b^1),\tilde{\B}(b^1)\in\LL(W^{\gamma})$ by
\begin{eqnarray*} 
\begin{array}{rcl}
\A(a,b^0) u&:=&\d_tu+a\d_xu+b^0u,\nonumber\\
\tilde\A(a,b^0) u&:=&-\d_tu-\d_x(au)+b^0u,\nonumber\\
\B(b^1) u&:=&b^1u,\nonumber\\
\tilde\B(b^1) u&:=&(b^1)^Tu.\nonumber
\end{array}
\end{eqnarray*}
Remark that the operators  $\A(a,b^0)$,
$\B(b^1)$,
and $\tilde{\B}(b^1)$ are well-defined for $a_j, b_{jk} \in  L^\infty(0,1)$, while 
$\tilde{\A}(a,b^0)$ is well-defined under additional regularity assumptions with respect to the coefficients
$a_j$, for example, 
for $a_j \in C^{0,1}([0,1])$.
Obviously, the operator equation 
\begin{equation}
\label{abstr}
\A(a,b^0)u+\B(b^1)u=f
\end{equation} 
is an abstract representation of the
periodic-Dirichlet problem~(\ref{eq:1.1})--(\ref{eq:1.3}).

Finally,  for $s \in \Z$
we introduce
the following complex $(n-m)\times(n-m)$ matrices 
\begin{equation}
\label{R}
R_s(a,b^0,r):=
\left[\sum\limits_{l=1}^m
e^{is(\al_j(1)-\al_l(1))+\be_j(1)-\be_l(1)}r_{jl}^1r_{lk}^0\right]_{j,k=m+1}^n,
\end{equation}
where
\begin{equation}\label{coef}
\al_j(x):=\int\limits_0^x\frac{1}{a_j(y)}\,dy, \;\;\be_j(x):=\int\limits_0^x
\frac{b_{jj}(y)}{a_j(y)}\,dy.
\end{equation}\\

Our first result concerns an isomorphism property of  $\A(a,b^0)$:
\begin{thm}\label{thm:1.0} 
For all $c>0$ there exists $C>0$ such that the following is true:
If 
\begin{equation}\label{ge}
a_j, b_{jj} \in  L^\infty(0,1)\; \mbox{ and } \; \ess\inf|a_j|\ge c \; \mbox{ for all } j=1,\ldots,n,
\end{equation}
\begin{equation}\label{le}
\sum_{j=1}^n\|b_{jj}\|_\infty
+ \sum_{j=1}^m\sum_{k=m+1}^n |r^0_{jk}|+\sum_{j=m+1}^n\sum_{k=1}^m |r^1_{jk}|\le \frac{1}{c},
\end{equation}
and
\beq
\label{cond}
|\det (I-R_s(a,b^0,r))|\ge c \; \mbox{ for all } \; s \in \Z,
\ee
then for all $\ga \ge 1$
the operator $\A(a,b^0)$ is an isomorphism from $V^{\gamma}(a,r)$ onto $W^{\gamma}$ and
$$
\|\A(a,b^0)^{-1}\|_{{\cal L}(W^{\gamma};V^{\gamma}(a,r))}\le C.
$$
\end{thm}
~\\

Our second result concerns the Fredholm solvability of (\ref{abstr}): 
\begin{thm}\label{thm:1.1} Suppose that conditions \reff{ge}  and \reff{cond}  are fulfilled for some $c>0$.
Suppose also that
\beq \label{eq:1.10}
\left.
\begin{array}{l}
\mbox{for all } j\ne k \mbox{ there is } c_{jk}\in BV(0,1) \mbox{ such that }\\
a_k(x)b_{jk}(x) = c_{jk}(x)(a_j(x)-a_k(x)) \mbox{ for a.a. } x \in [0,1].
\end{array}
\right\}
\ee
Then the following is true:

(i) The operator $\A(a,b^0)+\B(b^1)$ is a Fredholm operator with index zero
from $V^{\gamma}(a,r)$ into $W^{\gamma}$ for all $\ga\ge 1$, and 
$$
\ker({\A}(a,b^0)+{\B}(b^1)):=\left\{u \in V^{\gamma}(a,r):\,\left(\A(a,b^0)+\B(b^1)\right)u=0\right\}
$$
does not depend on $\ga$.

(ii) Suppose $a \in C^{0,1}\left([0,1];\M_n\right)$. Then 
\begin{eqnarray*}
\lefteqn{
\left\{\left(\A(a,b^0)+\B(b^1)\right)u: u \in V^{\gamma}(a,r)\right\}}\\
&&=\left\{f\in W^{\ga}:\langle f,u\rangle_{L^2}=0
\mbox{ for all }u\in\ker\left(\tilde \A(a,b^0)+\tilde \B(b^1)\right)\right\},
\end{eqnarray*}
where
$$
\ker({\tilde{\A}}(a,b^0)+{\tilde{\B}}(b^1)):=\{u \in \tilde{V}^{\gamma}(a,r):\,
\left(\tilde{\A}(a,b^0)+\tilde{\B}(b^1)\right)u=0\}
$$ 
does not depend on $\ga$.
\end{thm}
Here we write
\beq \label{eq:1.16}
\langle f,u\rangle_{L^2}:=\frac{1}{2\pi}\int\limits_0^{2\pi}\int\limits_0^{1}
\left\langle f(x,t),u(x,t)\right\rangle\,dxdt
\ee
for the usual scalar product in the Hilbert space $L^2\left((0,1)\times(0,2\pi);\R^n\right)$,
and $\langle \cdot,\cdot\rangle$ denotes the Euclidean scalar product in $\R^n$ (as well as the Hermitian 
scalar product in $\C^n$).

The main tools of the proofs of Theorems \ref{thm:1.0} and  \ref{thm:1.1} are separation of variables 
(cf. (\ref{eq:3.2})--(\ref{eq:3.3})), integral
representation of the solutions of the corresponding boundary value problems 
of the ODE systems (cf. (\ref{eq:3.14})),
and an abstract criterion for Fredholmness (cf. 
Lemma \ref{lem:4.1}).
In the special case $m=1,n=2,a_1(x)=1$, and $a_2(x)=-1$ Theorem \ref{thm:1.1} was proved in \cite{KR}.\\

Our last results concern the solution behavior of (\ref{abstr}) under small perturbations 
of the data $a$ and $b$ and  under arbitrary perturbations of $f$. In order to describe this
we use the following notation for the corresponding open balls (for $\varepsilon >0$):
\begin{eqnarray*}
A_\varepsilon(a):=\left\{\tilde{a} \in BV\left((0,1);\M_n\right): \tilde{a}=\mbox{diag}( \tilde{a}_1,\ldots, \tilde{a}_n),
\max_{1\le j\le n}\|\tilde{a}_{j}-a_{j}\|_\infty
<\varepsilon\right\},\\
B_\varepsilon^{\infty}(b):=\left\{\tilde{b} \in  L^\infty((0,1);\M_n): 
\max_{1\le j,k\le n} \|\tilde{b}_{jk}-b_{jk}\|_{\infty}<\varepsilon\right\},\\
B_\varepsilon(b):=\left\{\tilde{b} \in B_\varepsilon^{\infty}(b):\; 
\tilde{b}_{jk} \in BV(0,1) \mbox{ for all } 1 \le j \not= k \le n\right\}.
\end{eqnarray*}
The set $B_\eps^\infty(b)$ is open in the Banach space $L^\infty((0,1);\M_n)$. The sets 
$A_\eps(a)$ and  $B_\eps(b)$ will be considered as  open sets in the (not complete) normed vector spaces 
$\left\{\tilde{a} \in BV\left((0,1);\M_n\right): \tilde{a}=\mbox{diag}( \tilde{a}_1,\ldots, \tilde{a}_n)\right\}$ and 
$
\{\tilde{b} \in  L^\infty((0,1);\M_n): \;
\tilde{b}_{jk} \in BV(0,1) \mbox{ for all } 1 \le j \not= k \le n\}
$, equipped with the corresponding $L^\infty$-norms.

The solution behavior of (\ref{abstr}) under small perturbations 
of $b$ and $f$ follows directly from Theorem~\ref{thm:1.1} and the Implicit Function Theorem, because the map
\begin{equation}
\label{map}
b \in L^\infty\left((0,1);\M_n\right) \mapsto \left(\A(a,b^0),\B(b^1)\right) 
\in {\cal L}\left(V^{\gamma}(a,r);W^{\gamma}\right) \times {\cal L}(W^{\gamma})
\end{equation}
is affine and continuous:

\begin{cor}\label{corr}
Suppose  (\ref{ge}), (\ref{cond}) for some $c>0$, (\ref{eq:1.10}), and 
\begin{equation}
\label{dim}
\dim \ker (\A(a,b^0)+\B(b^1))=0. 
\end{equation}
Then there exists  $\varepsilon >0$ such that for all $\ga \ge 1$,
$\tilde{b} \in B_\varepsilon^\infty(b)$,  and
$f\in W^\ga$
there exists exactly one $u \in V^\ga(a,r)$ with 
$
\A(a,\tilde{b}^0)u+\B(\tilde{b}^1)u=f.
$
Moreover, the map 
$$
(\tilde{b},f) \in B_\varepsilon^\infty(b)\times W^\ga   
\mapsto u \in V^{\gamma}(a,r)
$$  
is $C^\infty$-smooth.
\end{cor}

In particular, Corollary \ref{corr} implies assertion (I) above, and, because of the continuous embedding 
$V^{\gamma}(a,r) \hookrightarrow C\left([0,1]\times[0,2\pi];\R^n\right)$ for $\ga >3/2$
(see Lemma~\ref{lem:2.2}(iii)), also assertion (II).\\

The solution behavior of (\ref{abstr}) under small perturbations of  $a$ and $r$ seems to be more complicated.
Under those perturbations the function spaces  $V^{\gamma}(a,r)$ change, in general. This makes them inappropriate. 
On the other hand, we don't know any Fredholmness results for the operator $\A(a,b^0)+\B(b^1)$ besides that which is 
described in Theorem~\ref{thm:1.1} and, hence, which is related to the choice of the function spaces $V^\ga(a,r)$ and $W^\ga$.

\begin{thm}\label{thm:1.2}
Suppose (\ref{dim}) and
\begin{equation}\label{ge1}
a_j \in BV(0,1), b_{jj} \in L^\infty(0,1), \mbox{ and } 
\inf|a_j| > 0 \; \mbox{ for all } j=1,\ldots,n,
\end{equation}
\begin{equation}
\label{ungleich}
b_{jk} \in BV(0,1) \mbox{ and } \inf |a_j-a_k|>0 \mbox{ for all } 1\le j\not=k \le n
\end{equation}
and
\beq
\label{condunif}
\sum_{j,k=m+1}^n \sum_{l=1}^m e^{2(\be_j(1)-\be_l(1))}|r_{jl}^1r_{lk}^0|^2<1.
\ee
Then  there exists $\varepsilon>0$ such that for all $\ga\ge 2$,
 $\tilde{a} \in A_\eps(a)$, $\tilde{b} 
\in  B_\varepsilon(b)$, and $f\in W^\ga$
there exists exactly one $u \in V^\ga(a,r)$ with 
$\A(\tilde{a},\tilde{b}^0)u+\B(\tilde{b}^1)u=f.$
Moreover, the map 
\begin{eqnarray}
\label{dts}
\lefteqn{
(\tilde{a},\tilde{b},f) \in A_\eps(a) 
\times  B_\varepsilon(b)\times W^\ga}  \nonumber \\&&
\mapsto u \in W^{\ga-k-1}\cap C\left([0,1]\times[0,2\pi];\R^n\right)
\end{eqnarray}
is $C^k$-smooth for all nonnegative integers $k\le \ga-1$.
\end{thm}

In particular, for $k=\ga-1$ (rsp. $k=\ga-2$)  we get assertion (III) above.\\

The present paper has been motivated mainly by two reasons:

The first reason is that the Fredholm property
of the linearization 
is  a key for many local investigations for nonlinear equations, such as small periodic forcing of stationary solutions 
to nonlinear autonomous problems (see, e.g. \cite{RePe})
or Hopf bifurcation (see, e.g. \cite{Ki,KR1}). In particular, those techniques are well established for 
nonlinear ODEs  and nonlinear parabolic PDEs, but almost nothing is known if those techniques 
work for nonlinear dissipative hyperbolic PDEs.

The second reason are applications to 
semiconductor laser dynamics \cite{LiRadRe,Rad,RadWu}.
Phenomena like Hopf bifurcation (describing the appearance of selfpulsations of lasers) 
and  periodic forcing of stationary solutions   
(describing the modulation of stationary laser states by time periodic electric pumping)
are essential for many applications of  semiconductor 
laser devices in communication systems (see, e.g., \cite{RadWu}).

Remark that our smoothness assumptions concerning $a_j$, $b_{jk}$, and $f_j(\cdot,t)$ are quite weak.
This is important for the applications to laser dynamics. But it turns out that
any stronger smoothness assumption with respect to the space variable $x$ would not essentially improve our results and would not simplify
the proofs.

Boundary value problems for hyperbolic systems of the type \reff{eq:1.1}, \reff{eq:1.3} are also used for modeling of correlated random walks
(see, e.g.  \cite{hillen,hiha,horst,Lutscher}).
\\

Our paper is organized as follows: 
In Subsection~\ref{sec:remarks} we comment about sufficient conditions for the key assumptions
\reff{cond},  \reff{eq:1.10}, \reff{dim}, and \reff{condunif} and about the question if those 
conditions as well as the assertions of  Theorems \ref{thm:1.0} and \ref{thm:1.1} are stable under small perturbations of the data.
In Section~\ref{sec:spaces} we introduce the main properties of the function spaces, used in this paper.
In Section~\ref{sec:iso} we prove  Theorem \ref{thm:1.0},
in Sections~\ref{sec:fredh} and  \ref{sec:index}   we prove Theorem \ref{thm:1.1}, and,
finally, in Section~\ref{sec:stability}   we prove Theorem \ref{thm:1.2}.

\subsection{Some comments}\label{sec:remarks}

\begin{rem}\label{rem:denom} {\bf about small denominators:}
\rm{
In Section \ref{sec:iso} we show the following: If one considers system \reff{eq:1.1}--\reff{eq:1.3} with vanishing nondiagonal coefficients, i.e.
with $b_{jk}=0$ for $j\not=k$, and if one makes a Fourier series ansatz for the solution, one ends up with 
linear algebraic systems for the vector valued Fourier coefficients. The system for the Fourier coefficient of order $s$ is uniquely solvable 
if and only if $\det(I-R_s(a,b^0,r))\not= 0$. In this case  $\det(I-R_s(a,b^0,r))$ appears in the demoninator of the
formula for the  Fourier coefficient.
The  condition \reff{cond} implies that the denominators are uniformly bounded from below, thereby
ensuring the convergence of the Fourier series. Using classical terminology, one can say that  \reff{cond}
allows us to avoid small denominators. 
}
\end{rem}

\begin{rem}\label{rem:sufficient} {\bf about the case \boldmath$m=1,\;n=2$\unboldmath: }
{\rm  In the case 
 $m=1, n=2$ the matrix $R_s(a,b^0,r)$ is the complex  number 
$$
R_s(a,b^0,r)=e^{is(\alpha_2(1)-\alpha_1(1))+\beta_2(1)-\beta_1(1)}r_{21}^1 r_{12}^0.
$$
Hence, in this case condition \reff{cond} is equivalent to
$$
e^{\beta_2(1)-\beta_1(1)}r_{21}^1 r_{12}^0 \not= 1.
$$
This fact was proved in our paper \cite{KR}.
For the cases  $n-m>1$ we don't know any  $s$-independent equivalent of condition \reff{cond}.
}
\end{rem}

\begin{rem}\label{rem:sufficient1} {\bf about a sufficient condition for \reff{cond}: }
{\rm 
Let us formulate, for general $m$ and $n$, a sufficient condition for  \reff{cond}, in which the parameter $s$
does not appear. 
Condition \reff{cond} is satisfied iff for all 
$s\in\Z$ the matrix $I-R_s(a,b^0,r)$ is invertible  and 
the operator norm $\|(I-R_s(a,b^0,r))^{-1}\|$ is bounded uniformly in $s\in\Z$. 
For that it is sufficient  to have 
\beq
\label{suf}
\|R_s(a,b^0,r)\|\le \mbox{const}<1 \mbox{ for all } s\in\Z.
\ee
Here we can use any operator norm in $\M_{n-m}$, corresponding to any norm in $\R^{n-m}$. If we take the Euclidean norm in 
$\R^{n-m}$, then the corresponding operator norm  in $\M_{n-m}$ can be estimated by the Euclidean norm in  $\M_{n-m}$.
In other words: \reff{suf} and, hence, \reff{cond} are satisfied if, for example, condition \reff{condunif} is satisfied.
This can be interpreted as a kind of control on small denominators via  parameters $a$, $b^0$, and $r$.
}
\end{rem}

\begin{ex}\label{ex:correlated} {\bf about a correlated random walk model: }
{\rm 
In the case  $m=1, n=2$ 
the sufficient for \reff{cond} condition \reff{suf} reads as
\beq
\label{suf2}
|r_{21}^1 r_{12}^0| \exp \int_0^1\left(\frac{b_{22}(x)}{a_2(x)}-\frac{b_{11}(x)}{a_1(x)}\right)dx<1.
\ee
Consider the following correlated random walk model
for chemotaxis (chemosensitive movement, see ~\cite{hirolu,segel}),
consisting of the hyperbolic system
$$
\left.
\begin{array}{rcl}
\partial_tu^+  + \partial_x\left(a^+(x)u^+\right)& =& -\mu^+(x)u^+ + \mu^-(x)u^-, \\
\partial_tu^-  - \partial_x\left(a^-(x)u^-\right)& =& -\mu^-(x)u^- + \mu^+(x)u^+, 
\end{array}
\right\}
x \in (0,1)
$$
with ``natural'' boundary conditions
$$
a^+(x)u^+(x,t) = a^-(x)u^-(x,t),\quad x=0,1.
$$
Translating the new notation to the old one, we get
$$
a_1=a^+, \;a_2=-a^-, \;b_{11}=\mu^++\partial_xa^+, \;b_{22}=\mu^--\partial_xa^-
$$
and 
$$
r_{12}^0=\frac{a^-(0)}{a^+(0)}, \;r_{21}^1=\frac{a^+(1)}{a^-(1)}.
$$
Therefore
$$
\exp\int_0^1\frac{b_{11}(x)}{a_1(x)} dx=\exp\int_0^1\frac{\mu^+(x)+\partial_xa^+(x)}{a^+(x)} dx
=\frac{a^+(1)}{a^+(0)} \exp\int_0^1\frac{\mu^+(x)}{a^+(x)}dx
$$
and analogously
$$
\exp\int_0^1\frac{b_{22}(x)}{a_2(x)} dx=\frac{a^-(0)}{a^-(1)} \exp\int_0^1\frac{\mu^+(x)}{a^+(x)}dx.
$$
Hence, condition \reff{suf2} is 
\beq\label{ne}
\int_0^1\left(\frac{\mu^+(x)}{a^+(x)}+\frac{\mu^-(x)}{a^-(x)}\right)\,dx > 0.
\ee
}
\end{ex}

\begin{rem}\label{rem:perturbcond} {\bf about small perturbations of the data in \reff{cond} and \reff{condunif}: }
{\rm
Let us comment about the behavior of the assumption \reff{cond} and its sufficient condition
\reff{condunif}
under small perturbations of the data.

If condition \reff{condunif} is satisfied for given data,
then it remains to be satisfied under sufficiently small perturbations of the coefficients $r_{jk}^0, r_{jk}^1$ and under 
sufficiently small (in $L^\infty(0,1)$) perturbations of the coefficient functions $a_j$ and $b_{jk}$.

If condition \reff{cond} is satisfied, then it remains to be satisfied under sufficiently small perturbations of $r_{jk}^0, r_{jk}^1$,
and $b_{jk}$, but not under small perturbations of  $a_j$, in general.
In other words,  \reff{cond} is not sufficient for \reff{condunif}.
It may happen that there exist arbitrarily small perturbations of $a_j$ that destroy the validity of
\reff{cond}:  

For example, consider the case $m=1, n=2, 
a_1(x)=\al, a_2(x)=-\al$, $b_{jk}(x)=0$ for $j,k=1,2$, $r_{1,2}^0=1, r_{2,1}^1=-1$.
Then \reff{cond} reads as 
\beq
\label{cond1}
|1+e^{\frac{2is}{\al}}| \ge \mbox{ const } >0 \mbox{ for all } s \in \Z.
\ee
This is satisfied iff
\beq
\label{I}
\al =\frac{2l+1}{k\pi} \mbox{ with } k \in \Z \mbox{ and } l \in \N.
\ee
In this case the set of all values $\al$ such that condition (\ref{cond}) is satisfied, is dense in $\R$, but 
the set of all values $\al$ such that (\ref{cond}) is not satisfied, is dense too.
}
\end{rem}

\begin{rem}\label{rem:perturbthm} {\bf about Fredholmness of  \boldmath${\cal A}(a,b^0)+{\cal B}(b^1)$\unboldmath
under small perturbations of the data: }
{\rm 
Let us comment about the behavior of the conclusions  of Theorem \ref{thm:1.1}, mainly the Fredholmness of 
the operator ${\cal A}(a,b^0)+{\cal B}(b^1)$,
under small perturbations of the data.

Suppose that for  given data $a$ and $b$ 
the assumptions of Theorem \ref{thm:1.1}
are satisfied. Then, under   sufficiently small perturbations of 
$b_{jk}$ in $L^\infty(0,1)$, independently whether  \reff{eq:1.10} remains to be true or not,
the Fredholmness of 
${\cal A}(a,b^0)$
survives because the map \reff{map} is continuous and because
the set of index zero Fredholm operators between two fixed Banach spaces is open. 

But if $a_j$, $r^0_{jk}$, or $r^1_{jk}$  are perturbed, then the function space $V^\ga(a,r)$ 
is changed, in general, and it may happen that there
exist arbitrarily small perturbations that destroy the Fredholmness: 

For example, consider again the case $m=1, n=2, 
a_1(x)=\al, a_2(x)=-\al, f(x)=0, b_{jk}(x)=0$ for $j,k=1,2$, $r_{1,2}^0=1, r_{2,1}^1=-1$.
Then \reff{cond} reads as \reff{cond1} which is equivalent to \reff{I}.
Hence, by Theorem \ref{thm:1.1}, if \reff{I} is true, then $\A(a,b^0)$ is Fredholm.
Condition \reff{I} and, hence, condition \reff{cond1} is not satisfied, for example, if
\beq
\label{II}
\al =\frac{2q}{(2p+1)\pi} \mbox{ with } p,q \in \N,
\ee
and in this case $\A(a,b^0)$ is not Fredholm because $\dim \ker \A(a,b^0) =\infty$:
Indeed, we have $(u_1,u_2) \in \ker \A(a,b^0)$ iff
\beq\label{III}
\partial_tu_1  + \al\partial_xu_1 = \partial_tu_2  - \al\partial_xu_2 =0, \;
x\in[0,1],  \; t \in \R, 
\ee
\beq\label{IV}
u_j(x,t+2\pi) = u_j(x,t),\; j=1,2,\;  x\in[0,1], \;t \in \R,
\ee
\beq
\label{IVa}
u_1(0,t) = u_2(0,t),\; u_2(1,t) = -u_1(1,t),
\; t \in \R.
\ee
The solutions of \reff{III} are of the type $u_1(x,t)=U_1\left(t-\frac{x}{\al}\right)$ and  $u_2(x,t)=U_2\left(t+\frac{x}{\al}\right)$. They satisfy \reff{IV} 
iff the functions $U_1$ and $U_2$ are $2\pi$-periodic. From the boundary condition in $x=0$ follows $U_1=U_2$, and, hence, the
boundary condition in $x=1$ reads as 
\beq
\label{V}
U_1\left(t-\frac{1}{\al}\right)=-U_1\left(t+\frac{1}{\al}\right).
\ee
Choosing
$
U_1(y)=\sin (r y), \; r \in \Z,
$
and using \reff{II}, condition \reff{V} transforms into
$$
\sin \left(r\left(t-\frac{2p+1}{2q}\pi\right)\right)=-\sin \left(r\left(t+\frac{2p+1}{2q}\pi\right)\right).
$$
This is fulfilled, for example, for $r=(2k+1)q$ and any choice of $k \in \Z$, i.e. 
we found infinitely many linearly independent solutions
to \reff{III}--\reff{IVa}.

The set of all values $\al$ of the type (\ref{II}) is dense in $[0,\infty)$. Hence, 
we get: In this case the set of all $\al >0$ such that
$\A(a,b^0)+\B(b^1)$ is Fredholm, is dense in $[0,\infty)$, but the  set of all $\al >0$ such that
$\A(a,b^0)+\B(b^1)$ is not Fredholm, is dense in $[0,\infty)$ too.
}
\end{rem}

\begin{rem}\label{rem:technical} {\bf about assumptions \reff{eq:1.10} and \reff{ungleich}:}
\rm{
Obviously,  the condition \reff{eq:1.10} is not necessary for the conclusions of Theorem \ref{thm:1.1} 
because the  conclusions of Theorem \ref{thm:1.1}  survive under small (in $L^\infty(0,1)$) perturbations of
the coefficients $b_{jk}$, but the   assumption  \reff{eq:1.10} does not, in general.

The following example shows that  Theorem \ref{thm:1.1} is not true, in general, if all its assumptions are fulfilled
with the exception of (\ref{eq:1.10}): Take $m=1,n=2, a_1(x)=a_2(x)=1, b_{11}(x)= b_{12}(x)= b_{22}(x)=f_1(x,t)=f_2(x,t)=0, b_{21}=b=$const.
Then (\ref{eq:1.1})--(\ref{eq:1.3}) looks like
\begin{eqnarray*}
\partial_tu_1  + \partial_xu_1  = 
\partial_tu_2  + \partial_xu_2 + bu_1  &=& 0,\\
u_1(x,t+2\pi)- u_1(x,t)=u_2(x,t+2\pi)-u_2(x,t)&=&0,\\
u_1(0,t) - r^0_{12}u_2(0,t)= 
u_2(1,t)-r^1_{21}u_1(1,t)&=&0.
\end{eqnarray*}
If  $r^0_{12}r^1_{21}<1$ and $b \not= 0$, then all  assumptions of Theorem \ref{thm:1.1} are fulfilled
with the exception of (\ref{eq:1.10}). If, moreover,
$$
b=\frac{r^0_{12}r^1_{21}-1}{r^0_{12}},
$$
then
$$
u_1(x,t)=\sin l(t-x), \; u_2(x,t)=b \left(\frac{1}{1-r^0_{12}r^1_{21}}-x\right)\sin l(t-x) , \; l \in \N,
$$
are infinitely many linearly independent solutions. Hence, the conclusion of  Theorem \ref{thm:1.1} is not true.

Finally, let us remark that, surprisingly,  the  assumption \reff{eq:1.10} 
is  used also in quite another circumstances,
for proving the spectrum-determined growth condition in $L^p$-spaces \cite{Guo,Luo,Neves} and in $C$-spaces \cite{Lichtner}
for semiflows generated by hyperbolic systems of the type  \reff{eq:1.1}, \reff{eq:1.3}.
}
\end{rem}

\begin{rem}\label{rem:injektiv} {\bf about sufficient conditions for  \reff{dim}: }
{\rm
Similarly to~\cite{kmit}, one can provide a wide range of   sufficient conditions for  \reff{dim}. 
We here concentrate on the physically relevant case
\begin{equation}
\label{kleinr} 
\sum\limits_{j=1}^m\sum\limits_{k=m+1}^n|r_{jk}^0|^2\le 1 \mbox{ and } \sum\limits_{j=m+1}^n\sum\limits_{k=1}^m|r_{jk}^1|^2\le 1.
\end{equation}
If \reff{eq:1.1}--\reff{eq:1.3}  with $f=0$ is satisfied, then
\begin{eqnarray}
\lefteqn{
\label{eq:u1}
0=\int\limits_{0}^{2\pi} \left(u_j^2(1,t)-u_j^2(0,t)\right)\,dt}\nonumber\\
&&+2 \int_0^{2 \pi} \int_0^1  a_j^{-1}(x)\left(b_{jj}(x)u_j^2+\sum_{k\not=j}
b_{jk}(x)u_ju_k\right) dx dt.
\end{eqnarray}
Using the reflection boundary conditions, summing up separately the first $m$ equations of (\ref{eq:u1}) 
and the rest $n-m$ equations 
of (\ref{eq:u1}), and subtracting the second resulting equality  from the first one, we get
\begin{eqnarray}
\label{eq:I}
\int\limits_{0}^{2\pi}\biggl(\sum\limits_{j=1}^mu_j^2(1,t)-\sum
\limits_{j=m+1}^n\left(\sum\limits_{k=1}^mr_{jk}^1u_k(1,t)\right)^2\nonumber\\
+\sum\limits_{j=m+1}^nu_j^2(0,t)-
\sum\limits_{j=1}^m\left(\sum\limits_{k=m+1}^nr_{jk}^0u_k(0,t)\right) ^2 \biggr)dt\nonumber\\
+2 \sum\limits_{j=1}^m\int_0^{2 \pi} \int_0^1  \frac{1}{a_j(x)}\left(b_{jj}(x)u_j^2+\sum_{k\not=j}
b_{jk}(x)u_ju_k\right) dx dt\nonumber\\
-2 \sum\limits_{j=m+1}^n\int_0^{2 \pi} \int_0^1  \frac{1}{a_j(x)}\left(b_{jj}(x)u_j^2+\sum_{k\not=j}
b_{jk}(x)u_ju_k\right) dx dt=0.
\end{eqnarray}
Applying H\"older's inequality and assumption (\ref{kleinr}), we derive that
\begin{eqnarray*}
\lefteqn{
\int\limits_{0}^{2\pi}\left(\sum\limits_{j=1}^mu_j^2(1,t)-\sum\limits_{j=m+1}^n
\left(\sum\limits_{k=1}^mr_{jk}^1u_k(1,t)\right)^2\right) dt} \\&&
\ge\left(1-\sum\limits_{j=m+1}^n\sum\limits_{k=1}^m|r_{jk}^1|^2\right)\int\limits_{0}^{2\pi}\sum\limits_{j=1}^mu_j(1,t)^2 dt\ge 0.
\end{eqnarray*}
A similar estimate is true for the second boundary summand in (\ref{eq:I}) as well. Set
\[
c_{jk}(x) := \frac{b_{jk}(x)}{a_j(x)} \mbox{ for } 1 \le j\le m\quad \mbox{ and }\quad c_{jk}(x) 
:= -\frac{b_{jk}(x)}{a_j(x)} \mbox{ for } m+1\le j\le n.
\]
Then  \reff{kleinr} together with 
\begin{eqnarray}
\label{xi}
\sum_{j,k=1}^nc_{jk}(x)\xi_j\xi_k\ge C\sum_{j=1}^n|\xi_j|^2 
\mbox{ for all } \xi\in\R^n \mbox{ and } \mbox{ a.a. } x\in(0,1),
\end{eqnarray}
where the constant $C>0$ does not depend on $\xi$ and $x$, is sufficient for \reff{dim}. 
It is easily seen that estimate (\ref{xi}) is true if, for instance,
\begin{eqnarray*}
\ess\inf\left\{\frac{b_{jj}}{a_j}-\sum\limits_{k\ne j}\left(\left|\frac{b_{jk}}{a_j}\right|
+\left|\frac{b_{jk}}{a_k}\right|\right)\right\}>0 &\mbox{ for all }& j=1,\ldots,m,\\
\label{coef2}
\ess\inf\left\{-\frac{b_{jj}}{a_j}-\sum\limits_{k\ne j}\left(\left|\frac{b_{jk}}{a_j}\right|
+\left|\frac{b_{jk}}{a_k}\right|\right)\right\}>0 &\mbox{ for all }& j=m+1,\ldots,n.
\end{eqnarray*}

Summarizing, we get: In order the main conditions \reff{dim} and \reff{condunif} to be satisfied, it is sufficient that
\reff{kleinr} is fulfilled as well as
\begin{eqnarray*}
\mbox{ess inf } a_j>0 & \mbox{for} & j=1,\ldots,m,\\
\mbox{ess sup } a_j<0 & \mbox{for} & j=m+1,\ldots,n,\\
\mbox{ess inf } b_{jj}>0 & \mbox{for} & j=1,\ldots,n,\\
\mbox{ess sup } |b_{jk}| \approx 0 & \mbox{for} & 1 \le j\not=k \le n.
\end{eqnarray*}
}
\end{rem}

\section{Some properties of the used function spaces}\label{sec:spaces}
\renewcommand{\theequation}{{\thesection}.\arabic{equation}}
\setcounter{equation}{0}

In this section we formulate some properties of the function spaces
$W^{\ga}$, $V^\ga(a,r)$, and  $U^{\ga}(a)$
introduced in Section~1. 
For each $u\in W^{\gamma}$ we have
\begin{equation}\label{eq:2.1}
u(x,t)=\sum\limits_{s\in\Z}u^s(x)e^{ist}\,\,\mbox{with}\,\,
u^s(x):=\frac{1}{2\pi}\int\limits_0^{2\pi}u(x,t)e^{-ist}\,dt,
\end{equation}
where $u^s \in L^2((0,1);\C^n)$, and the series
in~(\ref{eq:2.1}) converges to $u$ in the complexification of
$W^{\gamma}$. And vice versa: For any sequence $(u^s)_{s\in\Z}$
with
\begin{equation}
\label{eq:2.2}
u^s\in L^2((0,1);\C^n),\;\;u^{-s}=\overline{u^{s}},\;\; 
\sum\limits_{s\in\Z}(1+s^2)^{\ga}\|u^s\|_{L^2((0,1);\C^n)}^2<\infty
\end{equation}
there exists exactly one $u\in W^{\ga}$ with~(\ref{eq:2.1}). In what follows, we will identify
functions $u\in W^{\gamma}$ and sequences $\left(u^s\right)_{s\in\Z}$ with~(\ref{eq:2.2})
by means of~(\ref{eq:2.1}), and we will keep 
for the functions and the sequences
 the notations $u$ and $\left(u^s\right)_{s\in\Z}$, respectively.

The following lemma gives a compactness criterion in $W^{\ga}$ (see \cite[Lemma 6]{KR}):

\begin{lemma}\label{lem:2.1}
 A set $M\subset W^{\gamma}$ is precompact in $W^{\gamma}$ if and only if the 
following two conditions are satisfied:

(i) There exists $C>0$ such that for all $u\in M$ it holds
 $$
\sum\limits_{s\in\Z}(1+s^2)^{\gamma}\int\limits_0^1\|u^s(x)\|^2\,dx\le C.
$$

(ii) For all $\eps>0$ there exists $\de>0$ such that for all $\xi,\tau\in(-\de,\de)$ and
all $u\in M$ it holds
$$
\sum\limits_{s\in\Z}(1+s^2)^{\gamma}\int\limits_{0}^1\left\|u^s(x+\xi)e^{is\tau}-u^s(x)\right\|^2\,dx<
\eps,
$$
where $u^s(x+\xi):=0$ for $x+\xi\not\in[0,1]$.
\end{lemma}

Concerning the spaces $U^\ga(a)$ we have the following result:

\begin{lemma}\label{lem:2.2}
(i) The space $U^{\gamma}(a)$ is complete.\\
(ii) If $\ga \ge 1$, then for any $x \in [0,1]$ there exists  a continuous trace map
$u \in U^\ga(a) \mapsto u(x,\cdot)
\in L^2\left((0,2\pi);\R^n\right).$\\
(iii) If $\ga > 3/2$, then $U^\ga(a)$ is continuously embedded into $C([0,1] \times [0,2\pi];\R^n)$.
\end{lemma}

\begin{proof}
{\bf $(i)$}
Let $(u^{k})_{k\in\N}$ be a fundamental sequence
in $U^{\gamma}(a)$.
Then $(u^k)_{k\in\N}$ and $(\d_tu^k+a\d_xu^k)_{k\in\N}$
are fundamental sequences in $W^{\gamma}$.  This implies that $(\d_tu^k)_{k\in\N}$ and, hence,
$(a\d_xu^k)_{k\in\N}$ are fundamental sequences in $W^{\gamma-1}$. On the account of
$a_j \in L^\infty\left(0,1\right)$ and $\mbox{ess inf } |a_j|>0$ 
for all $j=1,\ldots,n$, the latter entails that $(\d_xu^k)_{k\in\N}$ is a fundamental sequence in $W^{\gamma-1}$
as well. Because $W^{\gamma}$ is complete for any $\ga$,
there exist
$u\in W^{\gamma}$ and $v,w\in W^{\gamma-1}$ such that
$$
u^k\to u \mbox{ in } W^{\gamma},\quad \d_tu^k\to v \mbox{ in } W^{\gamma-1}, \quad 
\d_xu^k\to w \mbox{ in } W^{\gamma-1} \quad \mbox{ as } k\to\infty.
$$
It is obvious that $\d_tu=v$ and $\d_xu=w$ in the sense of the generalized derivatives:
 Take a smooth function $\vphi: (0,1)\times\left(0,2\pi\right)\to\R^n$ with compact
support. Then
\begin{eqnarray*}
&\displaystyle
\int\limits_{0}^{2\pi}\int\limits_{0}^1\langle u,\d_t\vphi\rangle\,
dx\,dt=
\lim\limits_{k\to\infty}\int\limits_{0}^{2\pi}\int\limits_{0}^1
\langle u^k,\d_t\vphi\rangle\,dx\,dt&\\
&\displaystyle
=-\lim\limits_{k\to\infty}\int\limits_{0}^{2\pi}\int\limits_{
0}^1
\langle\d_tu^k,\vphi\rangle\,dx\,dt=-
\int\limits_{0}^{2\pi}\int\limits_{0}^1
\langle v,\vphi\rangle\,dx\,dt,&
\end{eqnarray*}
and similarly for $\d_x u$ and $w$. Hence $\d_tu+a\d_xu=v+aw$ in  $W^{\gamma-1}$. Since $(\d_tu^k+a\d_xu^k)_{k\in\N}$
is fundamental in $W^{\gamma}$, then $\d_tu+a\d_xu=v+aw$ in  $W^{\gamma}$ as desired.

Properties {\bf$(ii)$} and {\bf$(iii)$} can be proved similarly to \cite[Lemma 8 and Remark 9]{KR}.
\end{proof}

Now, let us consider the dual spaces $\left(W^{\ga}\right)^*$.

Obviously, for any $\ga \ge 0$ the spaces $W^{\ga}$ are densely 
and continuously embedded into the Hilbert space 
$L^2\left((0,1) \times(0,2\pi);\R^n\right)$.
 Hence, there is a canonical dense continuous embedding
\beq
\label{eq:2.4}
L^2\left((0,1) \times(0,2\pi);\R^n\right) \hookrightarrow \left(W^{\ga}\right)^*: 
[u,v]_{W^{\ga}} = \langle u,v \rangle_{L^2}
\ee
for all $u \in L^2\left((0,1) \times(0,2\pi);\R^n\right)$ and   $v \in W^{\ga}$.
Here $[\cdot,\cdot]_{W^{\ga}}:(W^{\ga})^* \times W^{\ga} \to \R$ is the dual pairing,
and $\langle \cdot,\cdot \rangle_{L^2}$ is the scalar product introduced in~(\ref{eq:1.16}).

Let us denote
\beq
\label{eq:e}
e_s(t):=e^{ist} \mbox{ for } s \in \Z \mbox{ and } t \in \R.
\ee
If a sequence $(\vphi^s)_{s \in \Z}$ with $\vphi^s \in L^2((0,1);\C^n)$ is given, then the pointwise
products $\vphi^se_s$ belong to $L^2\left((0,1)\times(0,2\pi);\C^n\right)$. 
Hence, they belong to the complexification of 
$(W^{\ga})^*$ (by means of the complexified version of (\ref{eq:2.4})), and it makes sense to ask if the 
series
\beq
\label{eq:2.5}
\sum_{s\in \Z} \vphi^se_s
\ee
converges in the complexification of $(W^{\ga})^*$. Moreover, we have (see \cite[Lemma 10]{KR})

\begin{lemma}\label{lem:2.5}
(i) \, For any $\vphi \in (W^{\ga})^*$ there exists a sequence $(\vphi^s)_{s \in \Z}$ with
\beq
\label{eq:2.6}
\vphi^s \in L^2((0,1);\C^n), \;\vphi^{-s}= \overline{\vphi^{s}}, \; 
\sum_{s\in \Z} (1+s^2)^{-\ga}\|\vphi^s(x)\|_{L^2((0,1);\C^n)}^2 < \infty,
\ee
and the series (\ref{eq:2.5}) converges to $\vphi$ in the complexification of $(W^{\ga})^*$.
Moreover, it holds
\beq
\label{eq:2.7}
\int\limits_{0}^1 \left\langle\vphi^s(x),u(x)\right\rangle \,dx =
\left[\vphi,ue_{-s}\right]_{W^{\ga}} \mbox{ for all  } s\in\Z \mbox{ and } u \in L^2\left((0,1);\R^n\right).
\ee

(ii) \, For any sequence $(\vphi^s)_{s \in \Z}$ with (\ref{eq:2.6}) the series (\ref{eq:2.5}) 
converges in the complexification of 
$(W^{\ga})^*$ to some $\vphi \in (W^{\ga})^*$, and (\ref{eq:2.7}) is satisfied.
\end{lemma}

\section{
Isomorphism property (proof of Theorem~\ref{thm:1.0})}\label{sec:iso}
\setcounter{equation}{0}

Let $\ga \ge 1$ and $f\in W^{\ga}$ be arbitrarily fixed. We have $f(x,t)=\sum\limits_{s\in\Z}f^s(x)e^{ist}$
with 
\begin{equation}\label{eq:3.1}
f^s\in L^2\left((0,1);\C^n\right), \;\; 
\sum\limits_{s\in\Z}(1+s^2)^{\gamma}\int\limits_{0}^1 \left\|f^s(x)\right\|^2\,dx <\infty.
\end{equation}
We have to show that, if
\reff{ge}, \reff{le}, and \reff{cond} hold, 
then there exists exactly one $u\in V^{\ga}(a,r)$ with
$$
\A(a,b^0) u=f \mbox{ and } \|u\|_{V^{\gamma}(a,r)} \le C \|f\|_{W^\ga},
$$
where the constant $C$ does not depend on $\ga, a, b^0,u,$ and $f$, but only on the constant $c$, which was 
introduced in the assumptions of Theorem \ref{thm:1.0}. But
$$
\|u\|_{V^{\gamma}(a,r)}=\|u\|_{W^{\gamma}}+\|\partial_tu+a\partial_xu\|_{W^{\gamma}}=\|u\|_{W^{\gamma}}+\|f-b^0u\|_{W^{\gamma}},
$$
hence we have to show that  there exists exactly one $u\in V^{\ga}(a,r)$ with
\beq
\label{prob}
\A(a,b^0) u=f \mbox{ and } \|u\|_{W^{\gamma}} \le C \|f\|_{W^\ga}
\ee
with a constant $C$, which  does not depend on $\ga, a, b^0,u,$ and $f$, but only on~$c$. 

Writing $u$  as series according to~(\ref{eq:2.1}) and~(\ref{eq:2.2}),  
it is easy to see that \reff{prob} is satisfied if   for all $s \in \Z$ we have
$u^s\in H^1\left((0,1);\C^n\right)$ and
\begin{equation}\label{eq:3.2}
a_j(x)\frac{d}{dx}u_j^s(x)+\left(is+b_{jj}(x)\right)u_j^s(x) = f_j^s(x),
\; j=1,\ldots,n,
\end{equation}
\begin{equation}\label{eq:3.3}
\left.
\begin{array}{rccl}
\displaystyle
u_j^s(0) &=& \displaystyle\sum\limits_{k=m+1}^nr_{jk}^0u_k^s(0),& j=1,\ldots,m,\\
\displaystyle
u_j^s(1) &=& \displaystyle\sum\limits_{k=1}^mr_{jk}^1u_k^s(1),& j=m+1,\ldots,n,
\end{array}\right\}
\end{equation}
\begin{equation}\label{eq:3.5}
\sum\limits_{s\in\Z}(1+s^2)^{\gamma}\int\limits_{0}^1|u_j^s(x)|^2\,dx\le C \|f\|_{W^\ga},\;\; j=1,\ldots,n.
\end{equation}
And vice versae: If  \reff{prob} is satisfied, then we have $u^s\in H^1\left((0,1);\C^n\right)$ and (\ref{eq:3.2})--(\ref{eq:3.5}). Indeed, 
take a smooth test function $\vphi:(0,1) \to \R$ with compact support. Then we have
\begin{eqnarray*}
\lefteqn{
\int_0^1\frac{f^s_j(x)\vphi(x)}{a_j(x)} dx=}\\
&&=\frac{1}{2\pi}\int_0^1\int_0^{2\pi}\frac{\left(\partial_tu_j(x,t)+a_j(x)\partial_xu_j(x,t)+b_{jj}(x)u_j(x,t)\right)\vphi(x)e^{-ist}}
{a_j(x)} dt dx\\
&&=\int_0^1\left(-u^s_j(x)\vphi'(x) +\frac{(is+b_{jj}(x))u^s_j(x)\vphi(x)}{a_j(x)}\right) dx.
\end{eqnarray*}
This implies $u_j^s\in H^1\left((0,1);\C\right)$ and \reff{eq:3.2}.
After that it follows easily that also the boundary condtions  \reff{eq:3.3} are fulfilled.

Now we are going  to show that there exists exactly one tuple of sequences
$(u_j^s)_{s\in\Z}, j=1,\ldots,n,$  with $u_j^s\in H^1\left((0,1);\C\right)$ 
satisfying~(\ref{eq:3.2})--(\ref{eq:3.5}).

By means of the variation of constants formula,~(\ref{eq:3.2}) is fulfilled if and only if
\begin{equation}\label{eq:3.7}
u_j^s(x) = e^{-is\al_j(x)-\be_j(x)}\left(u_j^s(0)+
\int\limits_0^xe^{is\al_j(y)+\be_j(y)}\frac{f_j^s(y)}{a_j(y)}\,dy\right),
\end{equation}
where the functions $\al_j$ and $\be_j$ are defined in (\ref{coef}).
The boundary conditions~(\ref{eq:3.3}) are satisfied if and only if
\begin{equation}\label{eq:3.9}
u_j^s(0)=\sum\limits_{k=m+1}^nr_{jk}^0u_k^s(0),\;j=1,\ldots,m,
\end{equation}
and
\begin{eqnarray*}
\label{eq:3.10}
\lefteqn{
e^{-is\al_j(1)-\be_j(1)}\left(u_j^s(0)+\int\limits_0^1e^{is\al_j(y)+\be_j(y)}\frac{f_j^s(y)}{a_j(y)}\,dy
\right)}\\
&&=\sum\limits_{k=1}^mr_{jk}^1e^{-is\al_k(1)-\be_k(1)}\left(u_k^s(0)+
\int\limits_0^1e^{is\al_k(y)+\be_k(y)}\frac{f_k^s(y)}{a_k(y)}\,dy
\right),\; \\
&&j=m+1,\ldots,n.
\end{eqnarray*}
This is equivalent to~(\ref{eq:3.9}),
\begin{eqnarray}\label{eq:3.12}
\lefteqn{
e^{-is\al_j(1)-\be_j(1)}u_j^s(0)-\sum\limits_{k=1}^m\sum\limits_{p=m+1}^n
e^{-is\al_k(1)-\be_k(1)}r_{jk}^1r_{kp}^0u_p^s(0)}\nonumber \\
&&=-e^{-is\al_j(1)-\be_j(1)}\int\limits_0^1e^{is\al_j(y)+\be_j(y)}\frac{f_j^s(y)}{a_j(y)}\,dy \nonumber\\
&&+\sum\limits_{k=1}^me^{-is\al_k(1)-\be_k(1)}r_{jk}^1
\int\limits_0^1e^{is\al_k(y)+\be_k(y)}\frac{f_k^s(y)}{a_k(y)}\,dy,\; \nonumber\\
&&j=m+1,\ldots,n.
\end{eqnarray}
The system~(\ref{eq:3.12}) has a unique solution
$(u_{m+1}^s(0),\dots,u_n^s(0))$ if and only if 
its coefficient matrix $I-R_s(a,b^0,r)$ (where $R_s(a,b^0,r)$ is introduced in (\ref{R}))
is regular. If, moreover, assumptions \reff{le}--(\ref{cond}) are satisfied, then
there exist coefficients
$c_{jk}^s$ and  a constant $C$ such that
\beq
\label{rep}
u_j^s(0)=\sum\limits_{k=1}^nc_{jk}^se^{-is\al_k(1)-\be_k(1)}
\int\limits_0^1e^{is\al_k(y)+\be_k(y)}\frac{f_k^s(y)}{a_k(y)}\,dy, \, j=m+1,\ldots,n,
\ee
and
$|c_{jk}^s|\le C$ uniformly with respect to  $a$, $b^0$, $r$, and $s\in\Z$
with  \reff{le}--(\ref{cond}).
Hence, for each $s\in\Z$ the 
boundary value problem~(\ref{eq:3.2})--(\ref{eq:3.3}) is uniquely solvable, and we
have the integral representation~(\ref{eq:3.7}) of the solution, where $ u_j^s(0)$ 
for $1\le j\le m$ is given by~(\ref{eq:3.9}) and for  $m+1\le j\le n$ by~(\ref{rep}).
Putting this together, we get
\begin{eqnarray}\label{eq:3.14}
\lefteqn{
u_j^s(x)=
e^{-is\al_j(x)-\be_j(x)}\Bigg(\int\limits_0^xe^{is\al_j(y)+\be_j(y)}\frac{f_j^s(y)}{a_j(y)}\,dy} \nonumber\\
&\displaystyle+\sum\limits_{k=1}^nd_{jk}^s
e^{-is\al_k(1)-\be_k(1)}\int\limits_0^1e^{is\al_k(y)+\be_k(y)}\frac{f_k^s(y)}{a_k(y)}\,dy\Bigg),
\; j=1,\ldots,n,&
\end{eqnarray}
with certain coefficients $d_{jk}^s$ such that there exists a constant $C$ (depending neither on $f$ nor on 
 $a$, $b^0$, $r$, and $s\in\Z$ satisfying  \reff{le}--(\ref{cond}))
with
\begin{equation}\label{eq:3.15}
|d_{jk}^s|\le C.
\end{equation}
In addition,~(\ref{eq:3.14}) and~(\ref{eq:3.15}) imply that there exists a constant $C$
(depending neither on $f$ nor on 
 $a$, $b^0$, $r$, and $s\in\Z$ satisfying  \reff{le}--(\ref{cond}))
such that
\begin{equation}\label{eq:3.16}
|u_{j}^s(x)|\le C\int\limits_0^1\|f^s(x)\|\,dx.
\end{equation}
The estimate~(\ref{eq:3.5}) now follows from (\ref{eq:3.1}).

\section{Fredholmness property (proof of Theorem~\ref{thm:1.1})}\label{sec:fredh}
\setcounter{equation}{0}

In Sections~\ref{sec:fredh} and~\ref{sec:index} we suppose the data  $a$, $b$, and $r$ to be fixed and to satisfy 
(\ref{cond})--(\ref{eq:1.10}). Hence we will omit the arguments in the operators and the spaces:
$$
\A:=\A(a,b^0),\; \B:=\B(b^1), \; \tilde{\A}:=\tilde{\A}(a,b^0),\; \tilde{\B}:=\tilde{\B}(b^1),
$$
$$
V^\ga:=V^\ga(a,r),\; \tilde{V}^\ga:=\tilde{V}^\ga(a,r).
$$

In this section we prove that
$\A+\B$ is Fredholm from $V^{\ga}$ into $W^{\ga}$, which is part of the assertions of 
Theorem~\ref{thm:1.1}. 

Obviously, $\A+\B$ is Fredholm from $V^{\ga}$ into $W^{\ga}$ if and
only if $I+\B\A^{-1}$ is Fredholm from $W^{\ga}$ into $W^{\ga}$. Here
$I$ is the identity in~$W^{\ga}$.

We will prove that  $I+\B\A^{-1}$ is Fredholm from $W^{\ga}$ into $W^{\ga}$
using the following abstract criterion for Fredholmness (see, e.g., \cite[Lemma 11]{KR} and \cite[Proposition 5.7.1]{Zeidler}):

\begin{lemma}\label{lem:4.1}
Let $W$ be a Banach space, $I$ the identity in $W$, and $\CC\in\LL(W)$ such that $\CC^2$
is compact.  Then $I+\CC$ is Fredholm.
\end{lemma}

In order to use Lemma~\ref{lem:4.1} with $W:=W^{\ga}$ and $\CC:= \B\A^{-1}$ we have to show that
$\left(\B\A^{-1}\right)^2$ is compact from $W^{\ga}$ into $W^{\ga}$.
For this purpose we will use 
Lemma~\ref{lem:2.1}. 

Condition $(i)$ of 
Lemma~\ref{lem:2.1} is satisfied because $\B\A^{-1}$ is a bounded operator from 
$W^{\ga}$ into $W^{\ga}$. 

It remains to check condition $(ii)$ of 
Lemma~\ref{lem:2.1}. For this purpose we will use the integral representation~(\ref{eq:3.14}) of $\A^{-1}$:

Take a bounded set $N\subset W^{\ga}$ and $f \in N$. Denote $u:=\A^{-1}f$ and $\tilde{u}:=
\left(\B\A^{-1}\right)^2f$. Then
\begin{eqnarray*}\label{tilde}
\tilde{u}_j^s(x)=\sum\limits_{k\ne j}
b_{jk}(x)e^{-is\al_k(x)-\be_k(x)}\Biggl(
\int\limits_0^xe^{is\al_k(y)+\be_k(y)}a_k^{-1}(y)\sum\limits_{l\ne k}
b_{kl}(y)u_l^s(y)\,dy\nonumber\\
+\sum\limits_{l=1}^nd_{kl}^se^{-is\al_l(1)-\be_l(1)}
\int\limits_0^1e^{is\al_l(y)+\be_l(y)}a_l^{-1}(y)\sum\limits_{r\ne l}
b_{lr}(y)u_r^s(y)\,dy
\Biggr).
\end{eqnarray*}
Therefore
$
\tilde{u}_j^s(x+\xi)e^{is\tau}-\tilde{u}_j^s(x)=P_j^s(x,\xi,\tau)+Q_j^s(x,\xi,\tau)+R_j^s(x,\xi)
$
with
\begin{eqnarray*}
P_j^s(x,\xi,\tau)&:=&
\sum\limits_{j\ne k \ne l}
\int\limits_x^{x+\xi}e^{is\left(-\al_k(x+\xi)+\tau+\al_k(y)\right)-\be_k(x+\xi)+\be_k(y)}\\[2mm]
&&\times a_k^{-1}(y)
b_{jk}(x+\xi)b_{kl}(y)u_l^s(y)\,dy,\\[2mm]
Q_j^s(x,\xi,\tau)&:=&\sum\limits_{k\ne j}
b_{jk}(x+\xi)e^{-\be_k(x+\xi)}\left(e^{is\left(-\al_k(x+\xi)+\tau\right)}-e^{-is\al_k(x)}\right)S_k^s(x),\\
R_j^s(x,\xi)&:=&\sum\limits_{k\ne j}
\left(b_{jk}(x+\xi)e^{-\be_k(x+\xi)}-b_{jk}(x)e^{-\be_k(x)}\right)S_k^s(x)
\end{eqnarray*}
and
\begin{eqnarray}
\label{S}
S_k^s(x):=
\int\limits_0^{x}e^{is\al_k(y)+\be_k(y)}a_k^{-1}(y)\sum\limits_{l\ne k}
b_{kl}(y)u_l^s(y)\,dy\nonumber\\
+ 
\sum\limits_{l=1}^nd_{kl}^se^{-is\al_l(1)-\be_l(1)}
\int\limits_0^1e^{is\al_l(y)+\be_l(y)}a_l^{-1}(y)\sum\limits_{r\ne l}
b_{lr}(y)u_r^s(y)\,dy.
\end{eqnarray}
We have to show that
\begin{eqnarray*}
\sum\limits_{s \in \Z}(1+s^2)^\ga\int_0^1\left(|P_j^s(x,\xi,\tau)|^2+|Q_j^s(x,\xi,\tau)|^2+|R_j^s(x,\xi)|^2\right)\,dx \to 0\\
\end{eqnarray*}
 for $|\xi|+|\tau| \to 0$ uniformly with respect to $f \in N$.

Because of $\A u=f$ we have (\ref{eq:3.16}). This implies 
\beq
\label{Pest}
\sum\limits_{s \in \Z}(1+s^2)^\ga\int_0^1|P_j^s(x,\xi,\tau)|^2\,dx \le
C \xi^2 \|f\|^2_{W^\ga},
\ee
where the constant $C$ does not depend on $j,\xi, \tau$, and $f$. Hence, the left hand side of 
(\ref{Pest}) tends to zero for  $|\xi|\to 0$ uniformly with respect to $f \in N$.
 
In order to estimate $Q_j^s(x,\xi,\tau)$ and $R_j^s(x,\xi)$, let us first estimate $S_j^s(x)$.
Again we use $\A u=f$. 
From~(\ref{eq:3.2}) it
follows
$$
\frac{d}{dy}\left(e^{is\al_l(y)}u_l^s(y)\right)=
e^{is\al_l(y)}\frac{f_l^s(y)-b_{ll}(y)u_l^s(y)}{a_l(y)}.
$$
Using this, we get 
\begin{eqnarray*}
\lefteqn{
is\frac{a_l(y)-a_k(y)}{a_k(y)a_l(y)}
e^{is\al_k(y)}u_l^s(y)}\nonumber\\
&&=e^{is(\al_k(y)-\al_l(y))}\frac{d}{dy}\left(e^{is\al_l(y)}u_l^s(y)\right)
-\frac{d}{dy}\left(e^{is\al_k(y)}u_l^s(y)\right)\nonumber\\
&&=e^{is\al_k(y)}\frac{f_l^s(y)-b_{ll}(y)u_l^s(y)}{a_l(y)}
-\frac{d}{dy}\left(e^{is\al_k(y)}u_l^s(y)\right).
\end{eqnarray*}
Therefore
\begin{eqnarray}
\label{eq:4.2}
\lefteqn{
e^{is\alpha_k(y)+\beta_k(y)} \frac{b_{kl}(y)}{a_k(y)} u_l^s(y)=\frac{ e^{\beta_k(y)}}{is} \frac{a_l(y)b_{kl}(y)}{a_l(y)-a_k(y)} }\nonumber\\
&&\times
\left(e^{is\alpha_k(y)} \frac{f_l^s(y)-b_{ll}(y)u_l^s(y)}{a_l(y)}-\frac{d}{dy}\left(e^{is\alpha_k(y)}u_l^s(y)\right)\right).
\end{eqnarray}
Moreover, because of assumption~(\ref{eq:1.10}), 
for all
$k\ne l$ the function
$$
y\in[0,1]\mapsto e^{\be_k(y)}\frac{a_l(y)b_{kl}(y)}{a_l(y)-a_k(y)}
$$
is in $BV(0,1)$. Hence,
\beq
\label{BVass}
\left|\int\limits_0^xe^{\be_k(y)}\frac{a_l(y)b_{kl}(y)}{a_l(y)-a_k(y)}
\frac{d}{dy}\left(e^{is\al_k(y)}u_l^s(y)\right)\,dy\right|\le 
C\|u_l^s\|_\infty,
\ee
the constant $C$ being independent of $x, k, l,s$, and $u$.  Therefore,~(\ref{eq:3.16})
and~(\ref{eq:4.2}) imply
\begin{equation}
\label{trick}
\left|\int\limits_0^xe^{is\al_k(y)+\be_k(y)}\frac{b_{kl}(y)}{a_k(y)}u_l^s(y)\,dy\right|\le \frac{C}{1+|s|}
\int\limits_0^1\|f^s(y)\|\,dy
\end{equation}
for some constant $C$ being independent of $x, k, l,s$, and $f$.
Similar estimates are true for all other integrals in~(\ref{S}).
As a consequence,
\beq\label{eq:4.3}
|S_j^s(x)|\le  \frac{C}{1+|s|}
\int\limits_0^1\|f^s(y)\|\,dy,
\ee
where $C$ does not depend on $x, j, s$, and $f$.
This gives
\begin{eqnarray*}
\lefteqn{
\int_0^1|Q_j^s(x,\xi,\tau)|^2\,dx }\\ 
&&\le
\frac{C}{1+|s|}  \max\limits_{k=1,\ldots,n}|e^{is\left(-\al_k(x+\xi)+\tau\right)}-e^{-is\al_k(x)}|^2
\int_0^1\|f^s(y)\|^2\,dy,
\end{eqnarray*}
where $C$ does not depend on $x,\xi,\tau, j, s$, and $f$.
But assumption (\ref{ge}) and notation (\ref{coef}) imply that
$$
|e^{is \left(-\al_k(x+\xi)+\tau\right)}-e^{-is\al_k(x)}|\le Cs(|\xi|+|\tau|),
$$
hence
\beq
\label{Qest}
\sum\limits_{s \in \Z}(1+s^2)^\ga\int_0^1|Q_j^s(x,\xi,\tau)|^2\,dx \le
C (\xi^2+\tau^2)\|f\|^2_{W^\ga},
\ee
where the constants, again, do not depend on $j,k,\xi,\tau$, and $f$. Hence, the left hand side of 
(\ref{Qest}) tends to zero for $|\xi|+|\tau| \to 0$   uniformly with respect to $f \in N$.

Finally, (\ref{eq:4.3}) gives
\begin{eqnarray}
\label{Rest}
\lefteqn{
\sum\limits_{s \in \Z}(1+s^2)^\ga\int_0^1|R_j^s(x,\xi)|^2\,dx}\nonumber\\ 
&&\le
C  \max\limits_{k=1,\ldots,n}\int\limits_0^1|b_{jk}(x+\xi)e^{-\be_k(x+\xi)}-b_{jk}(x)e^{-\be_k(x)}|^2\,dx
\|f\|^2_{W^\ga},
\end{eqnarray}
where the constant $C$ does not depend on $j, \xi, \tau$, and $f$. Hence, the left hand side of 
(\ref{Rest}) tends to zero for  $|\xi| \to 0$ uniformly with respect to $f \in N$
because of the continuity in the mean of the functions
$
x \mapsto b_{jk}(x)e^{-\be_k(x)}.
$ \\

\section{
Fredholm alternative (still proof of Theorem~\ref{thm:1.1})}\label{sec:index}
\setcounter{equation}{0}

To finish the proof of the assertion  $(i)$ of Theorem~\ref{thm:1.1}, 
it remains to show that the index of the operator $I+\B\A^{-1}$ is zero.
This is a straightforward consequence of  Lemma~\ref{lem:4.1} and a  homotopy argument:
Since $\left(\B\A^{-1}\right)^2\in\LL(W^\ga)$ is a compact operator,
the operators $\left(s\B\A^{-1}\right)^2\in\LL(W^\ga)$ are compact for any $s\in\R$ as well. By Lemma \ref{lem:4.1},
the operators $I+s\B\A^{-1}$ are Fredholm. Furthermore, they depend continuously on $s$. 
Since $I$ has index zero, the homotopy argument gives the same property for 
the operator $I+s\B\A^{-1}$ for any  $s\in\R$, in particular, for $s=1$.
Assertion  $(i)$ is thereby proved.

Summarizing, we proved  the Fredholm alternative for $\A+\B\in\LL(V^\ga,W^\ga)$. Hence, we have  
\beq
\label{Fr}
\left.
\begin{array}{r}
\dim \ker (\A+\B)=\dim \ker (\A+\B)^*<\infty,\\ 
\im(\A+\B)=\{f \in W^\ga: \; [\vphi,f]_{W^\ga}=0 \mbox{ for all } \vphi \in \ker(\A+\B)^* \}.
\end{array}
\right\}
\ee
Here $(\A+\B)^*$ is the dual operator  to $\A+\B$, i.e. a linear bounded 
operator from $(W^\ga)^*$ into $(V^\ga)^*$, and $[\cdot,\cdot]_{W^\ga}:\left(W^\ga\right)^* \times W^\ga \to \R$ is the dual pairing
in $W^\ga$.

To prove assertion  $(ii)$  of Theorem~\ref{thm:1.1},
we have to prove something slightly different, namely, that 
$$
\im(\A+\B)=\{f \in W^\ga: \; \langle f,u \rangle_{L^2}=0 \mbox{ for all } u \in \ker(\tilde{\A}+\tilde{\B}) \}
$$
and that $\ker(\A+\B)$ and $\ker(\tilde{\A}+\tilde{\B})$ do not depend on $\ga$.
Here $\langle \cdot,\cdot\rangle_{L^2}$ is the scalar product in $W^{0}=L^2\left((0,1)\times
\left(0,2\pi\right);\R^n\right)$ introduced in (\ref{eq:1.16}).

Directly from the definitions of the operators $\A$, $\tilde{\A}$, $\B$, and $\tilde{\B}$ it follows
\begin{eqnarray}
\label{eq:5.1}
\langle (\A+\B)u,\tilde{u} \rangle_{L^2} 
=\langle u,(\tilde{\A}+\tilde{\B})\tilde{u} \rangle_{L^2} 
\mbox{ for all } u \in V^\ga \mbox{ and } \tilde{u} \in \tilde{V}^\ga.
\end{eqnarray}
Using the continuous dense embedding (cf. \reff{eq:2.4})
$
\tilde{V}^\ga \hookrightarrow W^\ga  \hookrightarrow W^{0}  
\hookrightarrow \left(W^\ga\right)^*,
$
it makes sense to compare the subspaces $\ker (\A+\B)^*$ of $\left(W^\ga\right)^*$ 
and $\ker (\tilde{\A}+\tilde{\B})$ of $\tilde{V}^\ga$:

\begin{lemma}\label{lem:adj}
$\ker (\A+\B)^*=\ker (\tilde{\A}+\tilde{\B}).$
\end{lemma}

\begin{proof}
Because of (\ref{eq:2.4}) and (\ref{eq:5.1}), we have for all 
$u \in V^\ga$ and $\tilde{u} \in \tilde{V}^\ga$ that
$$
\langle (\tilde{\A}+\tilde{\B})\tilde{u},u\rangle_{L^2} =
\langle \tilde{u},(\A+\B)u\rangle_{L^2} =\left[\tilde{u},(\A+\B)u\right]_{W^\ga}=
\left[(\A+\B)^*\tilde{u},u\right]_{W^\ga}.
$$
This implies
$
\ker(\tilde{\A}+\tilde{\B})\subseteq \ker (\A+\B)^*.
$

Now, take an arbitrary $\vphi \in \ker(\A+\B)^*$ and  show that
$\vphi \in \ker(\tilde{\A}+\tilde{\B})$.
By  Lemma \ref{lem:2.5}, we have (using notation (\ref{eq:e}))
$
\vphi=\sum\limits_{s \in \Z} \vphi^se_s
$
with
$$
\vphi^s \in L^2\left((0,1);\C^n\right), \; \sum\limits_{s \in \Z}(1+s^2)^{-\ga}
\left\|\vphi^s(x)\right\|_{L^2((0,1);\C^n)}^2 < \infty.
$$ 
It follows  that for all $u \in V^\ga$
\begin{eqnarray*}
\lefteqn{
0=\left[(\A+\B)^*\vphi,u\right]_{W^\ga}=
\left[\vphi,(\A+\B)u\right]_{W^\ga}}\\&&=
\sum\limits_{s \in \Z} \int\limits_0^1 \left\langle\vphi^s,
a(x)\frac{d}{dx}u^{-s}-isu^{-s}+b(x)^Tu^{-s}\right\rangle\,dx.
\end{eqnarray*}
Therefore
$$
 \int\limits_0^1\left\langle\vphi^s,
a(x)\frac{d}{dx}u^{-s}-isu^{-s}+b(x)^Tu^{-s}\right\rangle\,dx=0
$$
for all $u^s \in H^1\left((0,1);\C^n\right)$ with (\ref{eq:3.3}). Since $a \in C^{0,1}\left([0,1];M_n\right)$,
by a standard argument, 
we conclude that
$\vphi^s \in H^1\left((0,1);\C^n\right)$ and that it satisfies the differential equation
\beq
\label{eq:adjGl}
-a(x)\frac{d}{dx}\vphi^{s}+\left(-is+b(x)^T - \frac{d}{dx}a(x)\right)\vphi^{s}=0
\ee
and the boundary conditions
\beq
\label{eq:adjBC}
\left.
\begin{array}{l}
\displaystyle
a_j(0)\vphi_j^s(0) = -\sum\limits_{k=1}^mr_{kj}^0a_k(0)\vphi_k^s(0),\qquad m+1\le j\le n,\\
\displaystyle
a_j(1)\vphi_j^s(1) = -\sum\limits_{k=m+1}^nr_{kj}^1a_k(1)\vphi_k^s(1),\qquad 1\le j\le m.
\end{array}
\right\}
\end{equation}
In other words: The functions $\vphi^{s}(x)e^{ist}$ 
belong to $\ker(\tilde{\A}+\tilde{\B})$ and, hence, to $\ker(\A+\B)^*$. 
But they are linearly independent, and $\dim \ker(\A+\B)^*<\infty$,
hence there is $s_0 \in \N$ such that $\vphi^{s}=0$ for  $|s|>s_0$. Therefore,  $\vphi \in \tilde{V}^\ga$ for all $\ga \ge 1$ and
$(\tilde{\A}+\tilde{\B})\vphi=0$ as desired. 
\end{proof}
As it follows from the proof of Lemma~\ref{lem:adj},
$$
\ker(\tilde{\A}+\tilde{\B})=\left\{\sum\limits_{|s|\le s_0} \vphi^se_s\,\bigg|\,\vphi^s 
\mbox{ solves (\ref{eq:adjGl}), (\ref{eq:adjBC})}\right\}
$$
does not depend on $\ga$. By \reff{Fr}, $\ker(\A+\B)$ does not depend on $\ga$ as well.
Claim (ii)  of Theorem \ref{thm:1.1} follows.

\section{$\Con^k$-smoothness of the data-to-solution map 
(proof of Theorem~\ref{thm:1.2})}\label{sec:stability}
\setcounter{equation}{0}
In this section we prove Theorem \ref{thm:1.2}. Hence, we suppose the assumptions of  Theorem~\ref{thm:1.2}
to be satisfied, i.e. the data $a,b$, and $r$, which satisfy  
(\ref{dim})--(\ref{condunif}), are given and fixed. 

Recall that in  Theorem \ref{thm:1.2} the sets $A_\eps(a)$ and $B_\eps(b)$ are the open balls around $a$ and $b$ of radius $\eps$ in the (not complete) normed vector spaces
$$
\Aa:=\{\tilde{a}=\mbox{diag}(\tilde{a}_1,\ldots,\tilde{a}_n) \in BV((0,1);\M_n)\} \mbox{ with } \|a\|_{\Aa}:=\max_{1 \le j \le n}\|\tilde{a}_j\|_\infty
$$
and 
\begin{equation}
\begin{array}{ll}
\Bb:=\{\tilde{b} \in  L^\infty((0,1);\M_n): \, \tilde{b}_{jk} \in BV(0,1) \mbox{ for all }1 \le j \not= k \le n\}\nonumber\\
 \mbox{ with } \|b\|_{\Bb}:=\max_{1 \le j,k \le n}\|\tilde{b}_{jk}\|_\infty,\nonumber
\end{array}
\end{equation}
respectively.

Because the assumptions of  Theorem \ref{thm:1.2}
are satisfied, there exists $\eps>0$ such that for all  $\tilde{a} \in \Aa$ and 
$\tilde{b} \in  \Bb$ with 
\beq
\label{kl}
\|\tilde{a}-a\|_{\Aa}+\|\tilde{b}-b\|_{\Bb}<\eps
\ee
the assumptions of  Theorem \ref{thm:1.1} are satisfied.
Therefore, for those $\tilde{a}$ and  $\tilde{b}$ the operators $\A(\tilde{a}, \tilde{b}^0)+\B(\tilde{b}^1)$
are Fredholm of index zero from $V^\ga(\tilde{a},r)$ into $W^\ga$ for all $\ga \ge 2$.

Moreover, all  $\tilde{a}=\mbox{diag}(\tilde{a}_1,\ldots,\tilde{a}_n)$ and  $\tilde{b} \in  L^\infty((0,1);\M_n)$ with \reff{kl} 
fulfill the assumptions of Theorem \ref{thm:1.0}. Hence, for those $\tilde{a}$ and $\tilde{b}$
and for all $\ga \ge 1$ the operator $\A(\tilde{a},\tilde{b}^0)$ is an isomorphism from
$V^{\ga}(\tilde{a},r)$ onto $W^{\ga}$  and
\begin{equation}
\label{const}
\|\A(\tilde{a},\tilde{b}^0)^{-1}\|_{{\cal L}(W^{\ga};V^{\ga}(\tilde{a},r))} \le C,
\end{equation}
where the constant $C>0$ does not depend on $\tilde{a}, \tilde{b}$, and $\ga$, but only on $\eps$.

For the sake of shortness, write $\CC:=C\left([0,1]\times[0,2\pi];\R^n\right)$.

\begin{lemma}\label{continuous}
{For all $\ga \ge 2$ the map $(\tilde{a},\tilde{b}^0) 
\mapsto  \A(\tilde{\om},\tilde{a},\tilde{b}^0)^{-1}$ is locally Lipschitz continuous 
as a map from a subset of $L^\infty((0,1);\M_n) \times L^\infty((0,1);\M_n)$ 
into $\LL(W^{\ga};W^{\ga-1}\cap \CC)$.}
\end{lemma}

\begin{proof}
Take  $a^\prime,a^{\prime\prime}\approx a$ and $b^\prime,b^{\prime\prime}\approx b^0$ 
in $L^\infty((0,1);\M_n)$,  $\ga \ge 2$,
$f \in W^{\ga}$, $u^{\prime} \in V^{\ga}(a^\prime,r)$, and  $u^{\prime\prime} 
\in V^{\ga}(a^{\prime\prime},r)$ such that
$$
\A(a^\prime,b^\prime)u^\prime=\A(a^{\prime\prime},b^{\prime\prime})u^{\prime\prime}=f.
$$ 
Then \reff{const} and Lemma \ref{lem:2.2}(iii) yield that there exists a constant $c>0$ such that
\begin{equation}
\label{const1}
c\|\partial_tu^\prime\|_{W^{\ga-1}}
\le c\|u^\prime\|_{W^{\ga}\cap \CC} \le 
\|u^\prime\|_{V^{\ga}(a^\prime,r)} \le C\|f\|_{W^\ga}
\end{equation}
and
\begin{eqnarray}
\label{const2}
&&c\min_{1\le j\le n} \ess\inf|a_{j}^\prime(x)|\; \|\partial_x u^\prime\|_{W^{\ga-1}} 
\le c\left\|a^\prime\partial_x u^\prime\right\|_{W^{\ga-1}}\nonumber \\
&&\le c\left\|\partial_t u^\prime +a^\prime\partial_x u^\prime\right\|_{W^{\ga-1}}
+c\|\partial_t u^\prime\|_{W^{\ga-1}}\nonumber \\
&&\le c\|u^\prime\|_{V^{\ga}(a^\prime,r)}+C\|f\|_{W^\ga}\le C(c+1)\|f\|_{W^\ga},
\end{eqnarray}
where the constant $C>0$ is the same as in \reff{const} 
and is independent of $\ga,a^\prime,a^{\prime\prime},b^\prime$, and $b^{\prime\prime}$.
Moreover, we have
\begin{eqnarray*}
&&\A(a^{\prime\prime},b^{\prime\prime})(u^{\prime\prime}-u^\prime)=\left(\A(a^{\prime\prime},b^{\prime\prime})-
\A(a^{\prime\prime},b^{\prime\prime})\right)u^\prime\\
&&+\left(\A(a^\prime,b^{\prime\prime})
-\A(a^{\prime\prime},b^{\prime\prime})\right)u^\prime+\left(\A(a^\prime,b^\prime)-
\A(a^\prime,b^{\prime\prime})\right)u^\prime\\
&&=\left(a^\prime-a^{\prime\prime}\right)
\partial_xu^\prime+\left(b^\prime-b^{\prime\prime}\right)u^\prime.
\end{eqnarray*}
This is a well-defined equation in $W^{\ga-1}$, and  \reff{const}, \reff{const1}, and \reff{const2} yield
\begin{eqnarray*}
\lefteqn{
\|u^{\prime\prime}-u^\prime\|_{W^{\ga-1}\cap \CC}}\\&& 
\le C\left(\max_{1 \le j \le n}\|a^{\prime\prime}_j-a^\prime_j\|_\infty+
\max_{1 \le j \le n}\|b^{\prime\prime}_{jj}-b^\prime_{jj}\|_\infty\right)\|f\|_{W^\ga},
\end{eqnarray*}
with a new constant $C$ being independent of $\ga,a^\prime,a^{\prime\prime},b^\prime,b^{\prime\prime}$, again. 
\end{proof}

\begin{lemma}\label{isomorph}
{There exists $\eps > 0$ with the following property:

For each  $\ga\ge 2$ there exists $C>0$ such that for all 
$\tilde{a} \in A_{\varepsilon}(a)$, and
$\tilde{b} \in B_{\varepsilon}(b)$ 
the operator $\A(\tilde{a},\tilde{b}^0)+\B(\tilde{b}^1)$ is an isomorphism from
$V^{\ga}(\tilde{a},r)$ onto $W^{\ga}$
and 
\begin{equation}
\label{const3}
\left\|\left(\A(\tilde{a},\tilde{b}^0)
+\B(\tilde{b}^1)\right)^{-1}\right\|_{{\cal L}(W^{\ga};V^{\ga}(\tilde{a},r))} \le C.
\end{equation}
}
\end{lemma}

\begin{proof}
As mentioned above, there exists $\varepsilon>0$ such that for all
$\tilde{a} \in A_{\varepsilon}(a)$,
and $\tilde{b} \in B_{\varepsilon}(b)$ 
the operator $\A(\tilde{a},\tilde{b}^0)+\B(\tilde{b}^1)$ is 
Fredholm of index zero from
$V^{\ga}(\tilde{a},r)$ into $W^{\ga}$.

Let us show that $\eps$ can be chosen so small that for all
$\tilde{a} \in A_{\varepsilon}(a)$,
and $\tilde{b} \in B_{\varepsilon}(b)$ 
the operator $\A(\tilde{a},\tilde{b}^0)+\B(\tilde{b}^1)$ is injective (and, hence, bijective).

Suppose the contrary.
Then there exist sequences  $\al_k \to a$ in $L^\infty((0,1);\M_n)$,
$\be_k \to b$ in $L^\infty((0,1);\M_n)$,
and $v_k\in V^\ga(\al_k,r)$
such that
\begin{equation}
\label{Gleichung}
(\A(\al_k,\be_k^0)+\B(\be_k^1))v_k = 0 \mbox{ and }  v_k\ne 0.
\end{equation}
Set $w_k:=\A(\al_k,\be_k^0)v_k$. Then there is $k_0\in\N$ such that for all $k\ge k_0$
\begin{equation}
\label{gleich}
(I+\B(\be_k^1)\A(\al_k,\be_k^0)^{-1})w_k = 0.
\end{equation}
Hence
$$
w_k=\left(\B(b_k^1)\A(\al_k,\be_k^0)^{-1}\right)^2w_k,
$$
and, consequently,
\begin{eqnarray}
\label{w1}
\lefteqn{
z_k:=\frac{w_k}{\|w_k\|_{W^{\ga}}}=\left(\B(b^1)\A(a,b^0)^{-1}\right)^2z_k}\nonumber\\
&&+\left(\left(\B(\be_k^1)\A(\al_k,\be_k^0)^{-1}\right)^2-\left(\B(b^1)\A(a,b^0)^{-1}\right)^2\right)z_k.
\end{eqnarray}
Because the operator $\left(\B(b^1)\A(a,b^0)^{-1}\right)^2$ is compact from $W^{\ga}$ into  $W^{\ga}$, it is also compact 
from $W^{\ga}$ into  $W^{\ga-1}$. Hence
there exist $z \in  W^{\ga-1}$ and a subsequence $z_{k_l}$ such that $\left(\B(b^1)\A(a,b^0)^{-1}\right)^2z_{k_l} \to z$ in 
$W^{\ga-1}$. Therefore Lemma~\ref{continuous} and  \reff{w1} yield that $z_{k_l} \to z$ in 
$W^{\ga-1}$, and  Lemma~\ref{continuous} and  \reff{gleich} yield that
$$
(I+\B(b^1)\A(a,b^0)^{-1})z= 0 \mbox{ in } W^{\ga-1}.
$$
Hence $z$ belongs to $\ker\left(I+\B(b^1)\A(a,b^0)^{-1}\right)$. Since $\|z\|_{W^{\ga}}=1$, 
we get a contradiction to assumption \reff{dim}.

It remains to prove \reff{const3}. 
Suppose the contrary, i.e. that for any choice of $\eps$ \reff{const3} is not true.
Then there exist sequences  $\al_k \to a$ in $L^\infty((0,1);\M_n)$,
$\be_k \to b$ in $L^\infty((0,1);\M_n)$,
and $v_k\in V^\ga(\al_k,r)$
such that 
\begin{equation}
\label{Ungleichung}
(\A(\al_k,\be_k^0)+\B(\be_k^1))v_k \to 0 \mbox{ in } W^\ga \mbox{ as } k\to\infty \mbox{ and } \|v_k\|_{V^\ga(\al_k,r)}=1.
\end{equation}
Now we proceed as above, replacing \reff{Gleichung} by \reff{Ungleichung}, to get a contradiction. 
\end{proof}

Similarly to Lemma~\ref{continuous} one can prove

\begin{lemma}\label{continuous1}
{For all $\ga \ge 2$ the map 
\begin{eqnarray*}
(\tilde{a},\tilde{b}) \in  A_{\eps}(a) 
\times B_{\varepsilon}(b)
\mapsto  \left(\A(\tilde{a},\tilde{b}^0)+\B(\tilde{b}^1)\right)^{-1}
\in\LL(W^{\ga};W^{\ga-1}\cap \CC)
\end{eqnarray*}
is locally Lipschitz continuous.}
\end{lemma}

Let us introduce the data-to-solution map
\begin{eqnarray}
\label{map1}
\lefteqn{
\left(\tilde{a},\tilde{b},f\right) \in
A_{\eps}(a) \times B_{\varepsilon}(b) \times W^{\ga}}\nonumber\\
&&\mapsto \hat{u}\left(\tilde{a},\tilde{b},f\right):=  
\left(\A(\tilde{a},\tilde{b}^0)+\B(\tilde{b}^1)\right)^{-1}f \in  V^{\ga}(\tilde{a},r).
\end{eqnarray}

\begin{lemma}\label{c1}
The map $\hat{u}$ is $C^1$-smooth as a map into $W^{\ga-2}\cap \CC$ for all $\ga\ge 3$.
\end{lemma}

\begin{proof} 
We have to show that all partial derivatives  $\partial_a \hat{u}$,
$\partial_b \hat{u}$, and  $\partial_f \hat{u}$ exist and  are continuous.

First, consider  $\partial_f \hat{u}$. From the definition \reff{map1} follows that it exists and that
$$
\partial_f\hat{u}\left(\tilde{a},\tilde{b},f\right)\bar{f}
=\left(\A(\tilde{a},\tilde{b}^0)+\B(\tilde{b}^1)\right)^{-1}\bar{f}.
$$ 
The continuity of the map
\begin{eqnarray*}
\left(\tilde{a},\tilde{b},f\right)\in
A_{\eps}(a) \times B_{\varepsilon}(b) \times W^{\ga}
\mapsto \partial_f\hat{u}\left(\tilde{a},\tilde{b},f\right) \in\LL(W^{\ga};W^{\ga-1}\cap \CC)
\end{eqnarray*}
follows from Lemma~\ref{continuous1}.

Now, consider  $\partial_b \hat{u}$. From Corollary~\ref{corr} follows that it exists, and \reff{map1} yields
\begin{eqnarray*}
\lefteqn{
\partial_b\hat{u}\left(\tilde{a},\tilde{b},f\right)\bar{b}=
-\left(\A(\tilde{a},\tilde{b}^0)+\B(\tilde{b}^1)\right)^{-1}}\\&&\times
\left(\partial_b\A(\tilde{a},\tilde{b}^0)\bar{b}+\B'(\tilde{b}^1)\bar{b}\right)
\left(\A(\tilde{a},\tilde{b}^0)+\B(\tilde{b}^1)\right)^{-1}f.
\end{eqnarray*}
Moreover, we have
\begin{equation}
\label{barB}
\left(\partial_b\A(\tilde{a},\tilde{b}^0)\bar{b}\right)u=\bar{b}^0u
\mbox{ and } \left(\B'(\tilde{b}^1)\bar{b}\right)u=\bar{b}^1u.
\end{equation}
Hence, $\partial_b\A(\tilde{a},\tilde{b}^0)$ and $\B'(\tilde{b}^1)$ do not depend on
$\tilde{a}$, and $\tilde{b}$. Therefore, again  Lemma~\ref{continuous1} yields the continuity of the map
$
\left(\tilde{a},\tilde{b},f\right)\in
 A_{\eps}(a) \times B_{\varepsilon}(b) \times W^{\ga}
\mapsto  \partial_b\hat{u}\left(\tilde{a},\tilde{b},f\right) 
\in\LL(\Bb;W^{\ga-1}\cap \CC).
$

Further, consider  $\partial_a \hat{u}$. If $\partial_a\hat{u}\left(\tilde{a},\tilde{b},f\right)$
exists as an element of the space 
$\LL(\Aa;W^{\ga-2}\cap \CC)$, then for any $\bar{a} \in BV((0,1),\M_n)$
we have
\begin{eqnarray}
\label{barA}
&&\left(\A(\tilde{a},\tilde{b}^0)+\B(\tilde{b}^1)\right) 
\partial_a\hat{u}\left(\tilde{a},\tilde{b},f\right)\bar{a}
=-\bar{a}\partial_x \hat{u}\left(\tilde{a},\tilde{b},f\right)\nonumber\\
&&=\bar{a}\tilde{a}^{-1}\left(\partial_t \hat{u}\left(\tilde{a},\tilde{b},f\right)
+\tilde{b}^1\hat{u}\left(\tilde{a},\tilde{b},f\right)-f\right).
\end{eqnarray}
The right hand side belongs to $W^{\ga-1}\cap \CC$, hence this 
equation determines uniquely the candidate 
$$
\hat{v}\left(\tilde{a},\tilde{b},f\right)\bar{a}:=
-\left(\A(\tilde{a},\tilde{b}^0)+\B(\tilde{b}^1)\right)^{-1}
\left(\partial_a \A(\tilde{a},\tilde{b}^0)\bar{a}\right)\hat{u}\left(\tilde{a},\tilde{b},f\right) 
$$
for $\partial_a\hat{u}\left(\tilde{a},\tilde{b},f\right)$
in  $\LL(\Aa;W^{\ga-1}\cap \CC)$.
Moreover, because of Lemma~\ref{continuous1} the candidate $\hat{v}$
for $\partial_a \hat{u}$ is continuous
as a map from  $ A_{\eps}(a) \times B_{\varepsilon}(b) \times W^{\ga}$
into  the space $\LL(\Aa;W^{\ga-2}\cap \CC)$.

It remains to prove that $\hat{v}$ is really  $\partial_a \hat{u}$. In order to show this, take 
  $a^\prime,a^{\prime\prime} \in   A_{\varepsilon}(a)$,
$\tilde{b} \in   B_{\varepsilon}(b)$,  $\ga \ge 2$,
$f \in W^{\ga}$, $u' \in V^{\ga}(a^\prime,r)$, and  $u'' \in V^{\ga}(a^{\prime\prime},r)$ such that
$$
\left(\A(a^\prime,\tilde{b}^0)+\B(\tilde{b}^1)\right)u'=\left(\A(a^{\prime\prime},\tilde{b}^0)+\B(\tilde{b}^1)\right)u''=f.
$$ 
Then
\begin{eqnarray*}
\lefteqn{
\left(\A(a^{\prime\prime},\tilde{b}^0)+\B(\tilde{b}^1)\right)\left(u''-u'-\hat{v}(a^\prime,\tilde{b}^0)(a^{\prime\prime}-a^\prime)\right)}\\
&&=\left(\A(a^\prime,\tilde{b}^0)-\A(a^{\prime\prime},\tilde{b}^0)\right)\hat{v}(a^\prime,\tilde{b}^0)(a^{\prime\prime}-a^\prime)\\&&
=-(a^{\prime\prime}-a^\prime)^2\d_x\hat{v}(a^\prime,\tilde{b}^0).
\end{eqnarray*}
Here we used that the map 
$\A(\cdot,\tilde{b})$
is affine. Hence
\begin{eqnarray*}
\left\|\left(\A(a^{\prime\prime},\tilde{b}^0)+\B(\tilde{b}^1)\right)\left(u''-u'-\hat{v}(a^\prime,\tilde{b}^0)(a^{\prime\prime}-a^\prime)\right)
\right\|_{W^{\ga-2}\cap \CC}\\
=o\left(\left\|a^{\prime\prime}-a^\prime\right\|_{\Aa}\right),
\end{eqnarray*}
and Lemma~\ref{isomorph} yields
$$
\left\|u''-u'-\hat{v}(a^\prime,\tilde{b}^0)(a^{\prime\prime}-a^\prime)\right\|_{W^{{\ga}-2}\cap \CC}
=o\left(\left\|a^{\prime\prime}-a^\prime\right\|_{\Aa}\right),
$$
\end{proof}

\begin{lemma}\label{ck}
The map $\hat{u}$ is $C^k$-smooth as a map into $W^{\ga-k-1}\cap \CC$ for all $1 \le k\le \ga-1$.
\end{lemma}

\begin{proof}
For $k=1$ the lemma is true, and the first partial derivatives satisfy
\begin{eqnarray}
\label{1}
\partial_f\hat{u}\left(\tilde{a},\tilde{b},f\right)\bar{f}&=& \hat{u}\left(\tilde{a},\tilde{b},\bar{f}\right),\\
\label{2}
\partial_b\hat{u}\left(\tilde{a},\tilde{b},f\right)\bar{b}&=& 
-\hat{u}\left(\tilde{a},\tilde{b},\bar{\B}
\left(\hat{u}\left(\tilde{a},\tilde{b},f\right),\bar{b}\right)\right),\\
\label{3}
\partial_a\hat{u}\left(\tilde{a},\tilde{b},f\right)\bar{a}&=& 
-\hat{u}\left(\tilde{a},\tilde{b},\bar{\A}\left(\tilde{a},\tilde{b},f,
\hat{u}\left(\tilde{a},\tilde{b},f\right)\right)\bar{a}\right).
\end{eqnarray}
Here we denoted by $\bar{\A}\left(\tilde{a},\tilde{b},f,u\right): \Aa\to W^{\ga-1}$ the linear bounded operator
which is defined by (cf.\reff{barA})
$$
\bar{\A}\left(\tilde{a},\tilde{b},f,u\right)\bar a:=
\bar{a}\tilde{a}^{-1}\left(\partial_t u
+\tilde{b}^1u-f\right),
$$
and $\bar{\B}:W^\ga \times \Bb\to W^{\ga}$ is the bilinear bounded operator which is defined by (cf. \reff{barB})
$$
\bar{\B}(u,b):=bu.
$$
Obviously, the map
\begin{eqnarray*}
\lefteqn{
\left(\tilde{a},\tilde{b},f,u\right) \in \R \times A_\eps(a) \times B_\eps(b) \times W^\ga \times W^\ga}\\
&&\mapsto \bar{\A}\left(\tilde{a},\tilde{b},f,u\right) 
\in \LL(\Aa;W^{\ga-1}\cap \CC)
\end{eqnarray*}
is $C^\infty$-smooth. Hence, \reff{1}--\reff{3}, Lemma~\ref{c1} and the chain rule imply that the data-to-solution map $\hat{u}$ is $C^2$-smooth,
and one gets corresponding formulae for the second partial derivatives by differentiating the identities  \reff{1}--\reff{3}. For example, it holds
$$
\begin{array}{l}
\partial^2_f\hat{u}\left(\tilde{a},\tilde{b},f\right)=0,\\
\partial_f\partial_b\hat{u}\left(\tilde{a},\tilde{b},f\right)\left(\bar{f},\bar{b}\right)\\
=\partial_b\hat{u}\left(\tilde{a},\tilde{b},\bar{f}\right)\bar{b}=
-\hat{u}\left(\tilde{a},\tilde{b},\bar{\B}
\left(\hat{u}\left(\tilde{a},\tilde{b},\bar{f}\right),\bar{b}\right)\right),\\
\partial_f\partial_a\hat{u}\left(\tilde{a},\tilde{b},f\right)\left(\bar{f},\bar{a}\right)\\
=\partial_a\hat{u}\left(\tilde{a},\tilde{b},\bar{f}\right)\bar{a}=
-\hat{u}\left(\tilde{a},\tilde{b},\bar{\A}\left(\tilde{a},\tilde{b},\bar{f},
\hat{u}\left(\tilde{a},\tilde{b},\bar{f}\right)\right)\bar{a}\right).
\end{array}
$$
If $\ga \ge 4$, then those formulae and the chain rule imply that all second partial derivatives of $\hat{u}$ are $C^1$-smooth etc.
\end{proof}

\section*{Acknowledgments}
The work of the  first  author was done while visiting the Humboldt University of Berlin under the support of the Alexander von Humboldt Foundation.
The second author acknowledges support of the DFG Research Center {\sc Matheon}
mathematics for key technologies (project D8).


\begin{thebibliography}{10}

\bibitem{Guo} Guo, B.-Z., G.-Q. Xu (2005). On basis property of a hyperbolic system with dynamic boundary condition.
{\it Differential Integral Equat.} 18: 35-60.

\bibitem{herrmann}Herrmann L.,  (1980). Periodic solutions of abstract differential equations: the
Fourier method. {\it Czechoslovak Math. J.}  30(105): 177--206.

\bibitem{hillen} Hillen, T. (1996). A turing model with correlated random walk. {\it J. Math. Biology} 35: 49-72.
               
\bibitem{hiha}  Hillen, T.,  Hadeler, K. P. (2005). {\it Hyperbolic systems and transport equations in mathematical biology.}
In: Analysis and 
Numerics for Conservation Laws, ed. by G. Warnecke. Berlin: Springer,  257--279.

\bibitem{hirolu}  Hillen, T., Rohde, C., Lutscher, F. (2001). Existence of weak solutions for a hyperbolic model of chemosensitive movement.
{\it J. Math. Anal. Appl.} 260: 173--199.


\bibitem{horst}  Horsthemke, W. (1999). Spatial instabilities in reaction random walks with direction-independent kinetics, 
{\it Phys. Rev E}  60: 2651-2663.


\bibitem{Ki}  Kielh\"ofer, H. (2004). {\it Bifurcation Theory. An Introduction with Applications
to PDEs}.
Appl. Math. Sciences {\bf 156}, Springer.


\bibitem{kmit} Kmit, I. (2007). Hyperbolic problems in the whole scale of Sobolev-type spaces
of periodic functions.
{\it Commentationes Mathematicae Universitatis Carolinae} 48,
No.\ 4: 631--645.


\bibitem{KR} Kmit, I., Recke, L. (2007). Fredholm alternative for periodic-Dirichlet problems for linear hyperbolic systems. 
{\it J. Math. Anal. Appl.} 335: 355--370.

\bibitem{KR1}  Kmit, I., Recke, L. (2011). Hopf bifurcation for semilinear hyperbolic systems with reflection boundary conditions. In preparation.

\bibitem{Lichtner} Lichtner, M. (2008). 
Spectral mapping theorem for linear hyperbolic systems. 
{\it Proc. Amer. Math. Soc.}  136, No. 6: 2091--2101. 


\bibitem{LiRadRe}  Lichtner, M., Radziunas, M., Recke, L. (2007).
Well-posedness, smooth dependence and center manifold reduction for a semilinear hyperbolic 
system from laser dynamics. {\it Math. Methods Appl. Sci.} 30: 931-960.


\bibitem{Luo} Luo Z.-H.,  Guo, B.-Z., Morgul,  O. (1999). {\it Stability and Stabilization of 
Infinite Dimensional systems with Applications.}
Berlin: Springer. 

\bibitem{Lutscher} Lutscher, F.,  Stevens, A. (2002). Emerging patterns in a hyperbolic 
model for locally interacting cell systems.
{\it J. Nonlinear Sci.}  12, No 6:  619-640.

\bibitem{Neves}  Neves, A. F.,  Ribeiro, H. De Souza,  Lopes, O. (1986). On the spectrum of evolution 
operators generated by hyperbolic systems.
{\it J. Functional Analysis} 67: 320--344.


\bibitem{Rad} Radziunas, M. (2006). Numerical bifurcation analysis 
of traveling wave model
of multisection semiconductor lasers. {\it Physica D}  213: 575--613.

\bibitem{RadWu} Radziunas, M., W\"unsche, H.-J. (2005).  Dynamics of 
multisection DFB semiconductor 
lasers: traveling wave and mode approximation models. {\it In: Optoelectronic 
Devices --
Advanced Simulation and Analysis}, ed. by J. ~Piprek, Berlin: Springer, 
 121--150.

\bibitem{RePe} Recke, L., Peterhof, D. (1998).  Abstract forced symmetry
breaking and forced
frequency locking of modulated waves, {\it J. Differ. Equat.}  144:
233--262.


\bibitem{robinson}  Robinson, J. C. (2001). {\it Infinite-Dimensional Dynamical Systems.}
Cambridge Texts in Appl. Math., Cambridge University Press.

\bibitem{segel} Segel, Lee A. (1977 )A theoretical study of receptor mechanisns in bacterial chemotaxis.
{\it SIAM J. Appl. Math.} 32:653--665.

\bibitem{vejvoda} Vejvoda, O. et al. (1981). {\it Partial Differential Equations: Time-Periodic Solutions.}
Sijthoff Noordhoff. 

\bibitem{Zeidler}  Zeidler, E. (1995). {\it Applied Functional Analysis. Main Principles and their Applications.}
Applied Math. Sciences 109,  Berlin: Springer.


\end{thebibliography}
\end{document}